\documentclass[11pt]{article}
\usepackage[margin=1in,footskip=0.25in]{geometry}
\usepackage{hyperref}
\usepackage{amsmath,amsfonts}
\usepackage{algorithmic}
\usepackage{algorithm}
\usepackage{array}
\usepackage[caption=false,font=normalsize,labelfont=sf,textfont=sf]{subfig}
\usepackage{textcomp}
\usepackage{stfloats}
\usepackage{url}
\usepackage{verbatim}
\usepackage{graphicx}
\usepackage{cite}
\usepackage{smile}
\usepackage[normalem]{ulem}
\usepackage{xcolor}
\usepackage[thinc]{esdiff}

\usepackage[shortlabels]{enumitem}

\begin{document}

\title{DISH: A Distributed Hybrid Optimization Method Leveraging System Heterogeneity}

\author{Xiaochun Niu \thanks{Xiaochun Niu is with the Department of Industrial Engineering and Management Sciences, Northwestern University, Evanston, IL 60201 USA (\url{xiaochunniu2024@u.northwestern.edu}). } \and Ermin Wei
\thanks{Ermin Wei is with the Department of Electrical and Computer Engineering and the Department of Industrial Engineering and Management Sciences, Northwestern University, Evanston, IL 60201 USA (\url{ermin.wei@northwestern.edu}).}}

\date{}
\maketitle

\begin{abstract}
We study distributed optimization problems over multi-agent networks, including consensus and network flow problems. Existing distributed methods neglect the heterogeneity among agents' computational capabilities, limiting their effectiveness. To address this, we propose DISH, a \underline{dis}tributed \underline{h}ybrid method that leverages system heterogeneity. DISH allows agents with higher computational capabilities or lower computational costs to perform local Newton-type updates while others adopt simpler gradient-type updates. Notably, DISH covers existing methods like EXTRA, DIGing, and ESOM-0 as special cases. To analyze DISH's performance with general update directions, we formulate distributed problems as minimax problems and introduce GRAND (\underline{g}radient-\underline{r}elated \underline{a}scent a\underline{n}d \underline{d}escent) and its alternating version, Alt-GRAND, for solving these problems. GRAND generalizes DISH to centralized minimax settings, accommodating various descent ascent update directions, including gradient-type, Newton-type, scaled gradient, and other general directions, within acute angles to the partial gradients. Theoretical analysis establishes global sublinear and linear convergence rates for GRAND and Alt-GRAND in strongly-convex-nonconcave and strongly-convex-PL settings, providing linear rates for DISH. In addition, we derive the local superlinear convergence of Newton-based variations of GRAND in centralized settings. Numerical experiments validate the effectiveness of our methods.
\end{abstract}

\section{Introduction}
We study distributed multi-agent optimization problems with communication constraints \cite{bertsekas1989parallel}. These include scenarios like \emph{distributed consensus problems} \cite{nedic2009distributed} and \emph{network flow problems}, driven by applications in power grids, sensor networks, communication networks, and machine learning \cite{lamnabhi2017systems,warnat2021swarm}. Agents in distributed computing are located at network nodes and restricted to local data and neighbor communication due to privacy and communication concerns. Their shared goal is to optimize an objective function collaboratively through distributed procedures.

There is a growing literature on developing distributed algorithms, such as gradient-type \cite{shi2015extra,nedic2017achieving,qu2017harnessing} and Newton-type methods \cite{shamir2014communication,zhang2015disco,wang2018giant,crane2019dingo} for consensus problems, and methods \cite{kelly1998rate} for network flow problems. However, a notable limitation of existing methods is that they often require all agents to take the same type of updates, leading to bottlenecks caused by agents equipped with slower hardware. This limitation restricts the applicability of fast-converging methods that rely on higher-order computations if even one agent in the system cannot handle them. Nonetheless, heterogeneous configurations are common in modern systems, where advanced processors coexist with older-generation ones,  resulting in agents with varying computation capabilities due to hardware constraints. Such heterogeneity presents significant challenges in practical distributed computing systems \cite{chen2015mxnet}. Thus, the question arises:

\vskip4pt
\emph{Can we design flexible and efficient distributed hybrid methods to utilize agents' heterogeneous computation capabilities?}

We answer this question affirmatively by proposing DISH, \underline{dis}tributed \underline{h}ybrid methods for consensus and network flow problems. DISH utilizes system heterogeneity by allowing agents to choose gradient-type or Newton-type updates based on their computation capabilities. In particular, there can be both gradient-type and Newton-type agents in the same communication round, and agents can switch between update types, adapting to their current situation. Notably, when all agents consistently perform Newton-type updates, the hybrid methods for the two problems provide distinct ways to approximate the dual Hessian in the Newton-type descent ascent method (NDA) with distributed implementations. For consensus problems, DISH covers well-known primal-dual gradient-type methods such as EXTRA \cite{shi2015extra}, DIGing \cite{nedic2017achieving}, and \cite{qu2017harnessing}, and primal-Newton-dual-gradient methods like ESOM-0 \cite{mokhtari2016decentralized} as special cases. It can also be applied to the dual problem of feature-partitioned distributed problems with efficient computation of the conjugate functions. Numerical experiments validate the effectiveness of our hybrid methods, showing faster convergence speeds as the number of Newton-type agents increases.

To analyze the performance of DISH with general update directions, we consider distributed applications as minimax problems and analyze the general \underline{g}radient-\underline{r}elated \underline{a}scent a\underline{n}d \underline{d}escent algorithmic framework (GRAND) for solving minimax problems. GRAND represents the generalization of DISH to centralized minimax settings. We introduce the  minimax optimization problem with $L:\RR^{d}\times\RR^{p}\to\RR$  strongly convex in $x$ but possibly nonconcave in $y$:
\#\label{prob:L}
\max_{y\in\RR^{p}}\min_{x\in\RR^{d}} L(x,y).
\#
The aforementioned distributed optimization problems can be formulated as a form of Problem \ref{prob:L}. Problem \ref{prob:L} has implications beyond distributed multi-agent optimization and is extensively studied in fields like supervised learning and adversarial training \cite{du2019linear}. The \emph{gradient descent ascent method} (GDA) is a simple method for tackling Problem \ref{prob:L}, which performs simultaneous gradient descent on $x$ and gradient ascent on $y$ at each iteration \cite{arrow1958studies,lin2020gradient}.
In addition, Newton-type methods with local superlinear convergence have been proposed \cite{tapia1977diagonalized,byrd1978local,zhang2020newton}.  However, existing analyses do not consider a mixture of first and second-order steps. This limitation, along with the demand for distributed hybrid methods, motivates the analysis of GRAND. GRAND allows $x$ and $y$ updates within uniformly bounded acute angles to $L$'s partial gradients. It covers GDA, scaled gradient, Newton-type, and quasi-Newton-type descent ascent methods as special cases. We also introduce the alternating version, Alt-GRAND, where $x$ and $y$ are updated sequentially. 

We establish the global sublinear convergence of GRAND and Alt-GRAND for strongly-convex-nonconcave problems. In addition, we demonstrate their linear convergence rates under the assumption of a strongly-convex-Polyak-\L{}ojasiewicz (PL) condition. This condition covers various scenarios, including distributed optimization problems, ensuring the linear rate of DISH.
The analysis faces challenges due to the coupled updates of $x$ and $y$ and the time-varying angles between updates and gradients. To tackle these challenges, we bound $y$'s optimality measure and $x$'s tracking error through coupled inequalities. Inspired by two-timescale analysis for bilevel problems, we consider linear combinations of these bounds as Lyapunov functions. Moreover, we examine the local performance of Newton-based methods in centralized settings. In particular, we show the local quadratic rates of the alternating Newton-type method (Alt-NDA) and its variants with multiple $x$ updates. In addition, we present a cubic-rate method that reuses the Hessian inverse for two consecutive steps. 

In summary, to the best of our knowledge, our distributed hybrid methods are the first to allow heterogeneous local updates for distributed consensus and network flow problems with provable convergence and rate guarantees.

\subsection{Related Works} 
Our work relates to the following growing literature.
\vskip4pt 
\noindent{\bf Distributed Optimization.} For \emph{distributed consensus problems} \cite{nedic2009distributed}, various first-order iterative methods exist. Distributed (sub)-gradient descent (DGD) \cite{nedic2009distributed} combines local gradient descent steps with weighted averaging among neighbors, achieving near-optimal solutions with constant stepsizes.
Other methods like EXTRA \cite{shi2015extra}, DIGing \cite{nedic2017achieving}, and \cite{qu2017harnessing} employ gradient tracking techniques and can be viewed as primal-dual gradient methods in augmented Lagrangian formulations, solving exact solutions with constant stepsizes. Second-order primal methods, such as Network Newton \cite{mokhtari2016network} and Distributed Newton method \cite{tutunov2019distributed}, approximate Newton steps iteratively through inner loops. Dual decomposition-based methods like ADMM \cite{boyd2011distributed}, ESOM \cite{mokhtari2016decentralized}, and PD-QN \cite{eisen2019primal} are also popular. Among them, PD-QN is a primal-dual quasi-Newton method with linear convergence. ESOM is closely related to our DISH method, which combines second-order primal updates with first-order dual updates and demonstrates provable linear convergence. However, none of these methods support heterogeneous agents with different update types. Our earlier work \cite{niu2023fedhybrid} develops a linearly converging distributed hybrid method allowing different update types but relying on the server-client (federated) network structure. 

In addition to the consensus problems (sample-partitioned), \emph{feature-partitioned} distributed problems are also prevalent in various fields, including bioinformatics, natural language processing, healthcare, and financial services \cite{boyd2011distributed}.

Distributed algorithms also tackle 
\emph{network flow optimization problems} \cite{bertsekas1989parallel} in fields like commodity networks and electric power systems. Existing literature covers first-order methods \cite{kelly1998rate} and second-order methods \cite{wei2013distributed}. 
Nonetheless, there is a lack of research exploring hybrid methods that enable different update types at network edges or agents.

\vskip4pt
\noindent{\bf Minimax Optimization.} 
A  simple method for minimax problems is GDA \cite{arrow1958studies}.
The monotonicity of the gradient $(\nabla_x L(x,y)^\intercal,-\nabla_y L(x,y)^\intercal)^\intercal$ enables analysis using theorems on monotone operators in variational inequalities \cite{nesterov2011solving}. 
Various first-order methods derived from GDA achieve better performance in different settings, like alternating GDA (Alt-GDA) \cite{bailey2020finite}. Second-order methods exploit Hessian information to accelerate convergence.  Some generalize Newton's method from minimization to minimax settings. Studies on Lagrangian problems corresponding to constrained optimization problems demonstrate superlinear local convergence \cite{tapia1977diagonalized,byrd1978local}. A complete Newton method with local quadratic rates is proposed \cite{zhang2020newton}. Cubic regularized Newton methods \cite{huang2022cubic} ensure global and local convergence rates by solving minimax subproblems at each iteration. However, global convergence analysis is lacking for Newton-type descent ascent methods without inner loops to solve subproblems or a line search to select stepsizes.

Methods like FR \cite{wang2019solving} and GDN \cite{zhang2020newton} involve first-order updates on $x$ with second-order updates on $y$. They converge locally to a minimax point, with GDN showing linear convergence. However, global performance analysis is missing for methods with general update directions.

\subsection{Contributions} As a summary, our contributions are as follows.
\begin{enumerate}
    \item We propose DISH, a hybrid method for consensus problems (also dual problems of feature-partitioned problems) and network flow problems. DISH leverages agents' heterogeneous computation capabilities by allowing them to choose between gradient-type and Newton-type updates.
    \item We establish global sublinear and linear rates for the GRAND frameworks for centralized strongly-convex-nonconcave and strongly-convex-PL minimax problems, ensuring linear convergence for DISH.
    \item We examine the local superlinear performance of Newton-based variations of GRAND for centralized minimax problems.
    \item Numerical results validate the efficiency of equipping the distributed systems with Newton-type agents.
\end{enumerate}

\subsection{Notation and Outline}
For a matrix $A$, let $\sigma_{\min}(A)$ ($\sigma_{\min}^+(A)$) be its smallest (non-zero) singular value. Let $\theta_{\min}(A)$ be the smallest eigenvalue if $A$ is positive definite. Let $\mathds{1}_n$ be the vector of all ones, $\otimes$ be the Kronecker product, and $\circ$ be the function composition. Let $O(\cdot)$ hide constants independent of the target parameter. 

The paper is organized as follows. Section \ref{sect:dist-opt-app} formulates distributed optimization problems as minimax problems and presents the distributed hybrid methods. Section \ref{sect:alg-GRAND} introduces GRAND and Alt-GRAND for centralized minimax problems. Section \ref{sect:global} analyzes the global convergence of GRAND. Section \ref{sect:local-Newton} discusses the local higher-order rates of Newton-based methods. Section \ref{sect:exp} demonstrates the numerical results.

\section{Hybrid Methods for Distributed Optimization} \label{sect:dist-opt-app}
In this section, we introduce distributed optimization problems, including consensus (DC) and network flow (NF) problems. We propose DISH as a distributed hybrid method. In particular, when all agents perform Newton-type updates, the methods for the two problem settings provide distinct ways to approximate the Newton-type descent ascent method (NDA) with distributed implementations.

We study optimization problems over multi-agent networks in both DC and NF settings. We define $\cG=\{\cN,\cE\}$ as a connected undirected network with the node set $\cN=\{1,\cdots,n\}$ and the edge set $\cE\subseteq \{\{i,j\}\given i,j\in\cN, i\neq j\}$. There are $n$ agents in the system, where each agent is located at a node of $\cG$ and can only communicate with its neighbors on $\cG$ due to privacy issues or communication budgets.

\subsection{Distributed Consensus Problems}
This section studies distributed consensus problems. We formulate the problem in a minimax form.

\subsubsection{Problem Formulation}
In consensus problems, all agents in the network aim to optimize an objective function collaboratively by employing a distributed procedure. Let $\omega\in\RR^d$ be the decision variable and $f_i:\RR^d\to \RR$ be the local function at agent $i$. We study an optimization problem over $\cG$ that $\min_{\omega} \sum_{i=1}^n f_i(\omega)$.
For example, for empirical risk minimization problems in supervised learning, $f_i$ is the empirical loss over local data samples kept at agent $i$. We impose the following standard assumptions on $f_i$.
\begin{assumption}
\label{ass:Hessian-dist}
The local function $f_i$ is twice differentiable, $m_i$-strongly convex, and $\ell_i$-Lipschitz smooth with constants $0 < {m_i} \le \ell_i <\infty$ for any agent $i\in\cN$.
\end{assumption}

Let $m_{\sf dc} = \min_{i\in\cN}\{m_i\}$ and $\ell_{\sf dc} = \max_{i\in\cN}\{\ell_i\}$. We decouple the computation of individual agent by introducing $x_i $ as the local copy of $\omega$ at agent $i$ to develop distributed methods. We formulate \textit{distributed consensus problems} \cite{nedic2009distributed} as
\#\label{prob:constrained_intro}
    \min_{x_1,\cdots,x_n\in \RR^d}   \sum_{i=1}^n f_i(x_i) \ \
    \text{s.t. } x_i = x_j, \text{ for } \{i,j\}\in\cE.
\#
The consensus constraints $x_i = x_j$ for $\{i,j\}\in\cE$ enforce the equivalence of  Problem \ref{prob:constrained_intro} and the original problem for a connected network $\cG$. For compactness, we denote by $x = (x_1^\intercal,\cdots,x_{n}^\intercal)^\intercal$ the concatenation of local variables and $f^{\sf dc}:\RR^{nd}\to\RR$ the aggregate function and reformulate Problem \ref{prob:constrained_intro} in an equivalent form,
\#\label{prob:constrained}
    \min_{x\in\RR^{nd}}  f^{\sf dc}(x) = \sum_{i=1}^n f_i(x_i) \quad
    \text{s.t. } (Z\otimes I_d) x = x, 
\#
where $Z\in \RR^{d \times d}$ is a nonnegative consensus matrix and satisfies the following assumption.

\begin{assumption}\label{ass:consensus} Matrix $Z$ corresponding to $\cG$  satisfies that
\begin{enumerate}
    \item[(a)] Off-diagonal elements: $z_{ij}\neq 0$ if and only if $\{i,j\} \in \cE$;
    \item[(b)] Diagonal elements: $z_{ii}>0$ for all $i\in\cN$;
    \item[(c)] $z_{ij}=z_{ji}$ for all $i\neq j$ and $i,j\in\cN$;
    \item[(d)] $Z\mathds{1}_n = \mathds{1}_n$.
\end{enumerate}
\end{assumption}
Assumption \ref{ass:consensus} is standard for consensus matrices. By Perron-Frobenius theorem, we have $\rho(Z)=1$, $\gamma<1$, 
and $\ker(I-Z) = \text{span}\{\mathds{1}_{n}\}$.  The matrix $Z$ ensures that $(Z\otimes I_d) x = x$ if and only if $x_i = x_j$ for all $\{i,j\}\in\cE$ \cite{nedic2009distributed}. 
Let $W=(I_{n} - Z)\otimes I_d $; thus $\rho(W)  < 2$, $\sigma_{\min}^+(W)=1-\gamma$, 
and $\ker(W) = \text{span}\{\mathds{1}_{n}\otimes y:y\in\RR^{d}\}$. We rewrite the constraint in Problem \ref{prob:constrained} as $Wx=0$.

Let variable $y = ( y_1^\intercal,\cdots,  y_n^\intercal)^\intercal$ represent the dual variable with $ y_i\in\RR^d$ associated with the constraint $z_{ii}x_i - \sum_{j\in\cN}z_{ij}x_j=0$ at agent $i$. We introduce the augmented Lagrangian $L^{\sf dc}(x, y)$ of Problem \ref{prob:constrained} with a constant $\mu\ge0$,
\#\label{eq:lagrangian_func}
L^{\sf dc}(x,y) = f^{\sf dc}(x) + y^{\intercal}W x + {\mu}x^{\intercal}Wx/2.
\#
The term $\mu x^{\intercal}Wx / 2$ is a penalty for violating the consensus constraint. The augmented Lagrangian in \eqref{eq:lagrangian_func} can also be viewed as the Lagrangian associated with a penalized problem 
$\min_{x}  f^{\sf dc}(x) + \mu{x}^{\intercal}W{x}/2$ such that $Wx=0$. It is equivalent to Problem \ref{prob:constrained} since $\mu x^{\intercal}Wx / 2$ is zero for any feasible $x$. By the convexity in Assumption \ref{ass:Hessian-dist} and Slater's condition,  strong duality holds for the penalized problem. Thus, the penalized problem and Problem \ref{prob:constrained} are equivalent to the dual problem,
\#\tag{DC} \label{prob:dual}
\max_{y\in\RR^{nd}} \psi^{\sf dc}(y), \text{ where } \psi^{\sf dc}(y) = \min_{x\in\RR^{nd}} L^{\sf dc}(x, y),
\#
where we refer to $\psi^{\sf dc}:\RR^{nd}\to\RR$ as the dual function and the problem as Problem \ref{prob:dual}.
We now develop distributed methods to solve Problem \ref{prob:dual}. As we will illustrate after Assumption \ref{ass:Hessian}, given any $y\in\RR^{nd}$, $L^{\sf dc}(\cdot,y)$ is strongly convex with a unique minimizer. For convenience, for any $L$ in Problem \ref{prob:L} satisfying such a condition, letting $x^*(y)$ be the unique minimizer for any $y$, we define $\psi:\RR^p\to\RR$ as follows,
\#\label{eq:x_star_y}
& x^*(y) = \argmin_{x\in\RR^d} L(x,y), \notag\\
& \psi(y) = \min_{x\in\RR^d} L(x,y) = L(x^*(y), y).
\#
The next lemma shows the forms of $\nabla \psi(y)$ and $\nabla^2 \psi(y)$, based on the well-known envelope theorem. We show it here for completeness. 
Let $N:\RR^{d}\times\RR^p\to\RR^{p\times p}$ be an operator,
\#\label{eq:N-operator}
N(x,y) = \nabla_{yx}^2 L(x,y) [\nabla_{xx}^2L(x,y)]^{-1}\nabla_{xy}^2L(x,y) - \nabla_{yy}^2 L(x,y). 
\#

\begin{lemma}\label{lem:psi-grad-hess}
    Given any $y\in\RR^p$, suppose $L(\cdot, y)$ is strongly convex with a unique minimizer $x^*(y)$. With $x^*(y)$ defined in \eqref{eq:x_star_y} and $N$ defined in \eqref{eq:N-operator}, it holds that
    $\nabla \psi(y) = \nabla_y L(x^*(y), y)$ and $\nabla^2 \psi(y) = - N(x^*(y),y)$.
\end{lemma}
   
Lemma \ref{lem:psi-grad-hess} shows that $-N$ can evaluate the Hessian $\nabla^2 \psi(y)$  with appropriate arguments. This property allows us to approximate $\nabla^2 \psi(y)$ in a distributed manner when designing the hybrid methods.

\subsubsection{Distributed Hybrid Methods for Consensus Problems}
We propose DISH to solve Problem \ref{prob:dual}. It allows choices of gradient-type and Newton-type updates for each agent at each iteration based on their current computation capabilities. The compact form of DISH shows as follows. At iteration $k$, 
\#\label{eq:up_pre}
 & x^{k+1} =  x^k - AP^k\nabla_{ x}L^{\sf dc}( x^k, y^k), \notag\\
 & y^{k+1}= y^k + BQ^k\nabla_{ y} L^{\sf dc}( x^k, y^k),
\#
where stepsize matrices $A = \diag\{a_1, \cdots, a_n\}\otimes I_d$ and $B = \diag\{b_1,\cdots,b_n\}\otimes I_d$ consist of personalized stepsizes $a_i$ and $b_i>0$ for $i\in\cN$ and block diagonal scaling matrices $P^k = \diag\{P_1^k,\cdots,P_n^k\}$ and $Q^k = \diag\{Q_1^k,\cdots,Q_n^k\}$ consist of positive definite local scaling matrices $P_i^k$ and $Q_i^k\in\RR^{d\times d}$ for $i\in\cN$. Here are some examples of possible scaling matrices:
\#\label{eq:dish-choices}
& \textit{Primal:}\quad \text{Gradient-type: }P_i^k=I_d; \nonumber\\
& \quad \quad \quad \text{Newton-type: }P_i^k = (\nabla^2 f_i(x_i^k) + \mu I_d)^{-1}. \nonumber\\
& \textit{Dual:}\quad \text{Gradient-type: }Q_i^k=I_d; \nonumber \\
& \quad \quad \quad \text{Newton-type: }Q_i^k = \nabla^2 f_i(x_i^k) + \mu I_d.
\#
We refer to \cite{niu2022dish} for a detailed explanation of the choices of local scaling matrices. 
We define two cases of DISH: DISH-G, where all agents perform gradient-type updates (which is equivalent to GDA), and DISH-N, which approximates the Newton-type descent ascent method (NDA) with a distributed procedure since the primal (for $\mu>0$) and dual Hessians are inseparable. In addition to gradient-type and Newton-type updates, DISH allows agents to take other local updates, such as scaled gradient or quasi-Newton directions. Algorithm \ref{alg:dish} presents the distributed implementation of DISH by substituting the partial gradients in \eqref{eq:up_pre}. It includes a primal step (Line \ref{line:primal}) and a dual step (Line \ref{line:dual}) at each agent. Moreover, an alternating version of DISH under the Alt-GRAND framework is ensured to converge. Another practical variant is when agent $i$ obtains $y_i^{k+1}$ using the updated $x_i^{k+1}$, some $x_j^{k}$, and some updated $x_j^{k+1}$ from its neighbors.

\begin{algorithm}[ht] 
    \caption{DISH for Consensus Problems}
    \begin{algorithmic}[1]\label{alg:dish}
    \STATE{\textbf{Input:} Initialization $x_i^0, y_i^0\in\RR^d$, stepsizes $a_i,b_i>0$ for $i\in \cN$, and $\mu\ge 0$.}  
    \FOR{$k =  0, \ldots, K-1$} 
        \FOR{each agent $i\in \cN$ in parallel}
            \STATE{Send $x_i^{k}$ and $ y_i^{k}$ to its neighbors $j$ for $\{i,j\}\in\cE$;}
            \STATE{Choose its local scaling matrices $P_i^k$ and $Q_i^k$;}
            \STATE{$x_i^{k+1} = x_i^k - a_iP_i^k[\nabla f_i(x_i^k) + (1-z_{ii})(y_i^k +\mu x_i^k) -\sum_{j\colon\{j,i\}\in\cE}z_{ij} (y_{j}^k+\mu x_{j}^k)];$ }\label{line:primal}
           \STATE{$ y_i^{k+1} =  y_i^k + b_iQ_i^k\big[(1-z_{ii})x_i^k - \sum_{j\colon\{i,j\}\in\cE}z_{ij}x_{j}^k\big]$.}\label{line:dual}
        \ENDFOR
    \ENDFOR
\end{algorithmic}
\end{algorithm}

DISH covers existing distributed methods such as EXTRA \cite{shi2015extra}, DIGing \cite{nedic2017achieving}, \cite{qu2017harnessing}, and ESOM-0 \cite{mokhtari2016decentralized} through appropriate parameter choices. More details on these relationships can be found in \cite{niu2022dish}. DISH allows agents with higher computational capabilities or cheaper computational costs to locally implement Newton-type updates, while others can adopt simpler gradient-type updates. It provides flexibility by allowing agents to use different types of updates across iterations and between primal and dual spaces within the same iteration. Numerical studies in Section \ref{sect:exp-dish} show that DISH achieves faster performance when more agents adopt Newton-type updates since it better utilizes local information. It is worth noting that Algorithm \ref{alg:dish} offers alternative ways to develop distributed methods beyond the choices in \eqref{eq:dish-choices}. For instance, PD-QN \cite{eisen2019primal}, which matches the linear rate of DISH, approximates the primal-dual quasi-Newton method using distributable matrices that satisfy the quasi-Newton (global secant) conditions. PD-QN is a special case of Algorithm \ref{alg:dish} since its scaling matrices are uniformly lower and upper-bounded.

\subsubsection{Feature-Partitioned Distributed Problems} 
We consider prediction problems over $\cG$ and denote by $\Theta \in \RR^{N\times d}$ the input data matrix with $N$ samples and $d$ features. Then Problem \ref{prob:constrained_intro} corresponds to sample-partitioned settings with partitioned data $\Theta = (\theta^\intercal_1,\cdots,\theta^\intercal_n)^\intercal$, where a row block $\theta_i\in\RR^{N_i\times d}$ represents the $N_i$ local samples kept at agent $i$ and $\sum_{i\in\cN} N_i=N$. Alternatively, in feature-partitioned settings \cite{boyd2011distributed}, the data matrix is split into $\Theta = (\Theta_1,\cdots,\Theta_n)$, where a column block $\Theta_i\in\RR^{N\times d_i}$ is the $d_i$ local features kept at agent $i$ and $\sum_{i\in\cN} d_i=d$. In this setting, each agent has access to the entire set of data samples but only a unique subset of the features. The previous section presents DISH to solve sample-partitioned consensus problems, and now we consider its extension to feature-partitioned distributed settings.

Feature-partitioned problems are likely to involve a moderate number of samples and a large number of features \cite{boyd2011distributed}. For example, scientists can collaboratively study DNA mutations using a few volunteers' DNA data recorded at multiple labs; and doctors may evaluate shared patients' health conditions by leveraging their medical data from several specialists.

Given partitioned data $\Theta = (\Theta_1,\cdots,\Theta_n)\in\RR^{N\times d}$ and $\Theta_i\in\RR^{N\times d_i}$, we decompose the decision variable as $\xi = (\xi_1^\intercal,\cdots,\xi_n^\intercal)^\intercal\in\RR^d$ with $\xi_i\in\RR^{d_i}$. This gives $\Theta \xi=\sum_{i\in\cN} \Theta_i \xi_i$. We consider a convex loss function $\phi$ and a convex and separable regularizer $r$ such that $r(\xi) = \sum_{i\in\cN} r_i(\xi_i)$.  Examples of separable regularizers include the $l_2$ norm $\|\xi\|^2 = \sum_{i\in\cN} \|\xi_i\|^2$. We can formulate the optimization problem of the feature-partitioned scenario as follows,
\#\label{prob:feature}
\min_{\xi\in\RR^d} \phi\big(\sum_{i=1}^n \Theta_i \xi_i\big) + \sum_{i=1}^n r_i(\xi_i).
\#
Let $f^*(\lambda) = \max_x\{\lambda^\intercal x - f(x)\}$ be the convex dual conjugate of any function $f$. The following proposition shows that the dual problem of Problem \ref{prob:feature} takes the form of the consensus problem in \eqref{prob:constrained}. Similar results are also shown in \cite{boyd2011distributed}.
\begin{proposition} \label{prop:feature}
Problem \ref{prob:feature} is equivalent to the following problem with $x=(x_1^\intercal,\cdots,x_n^\intercal)^\intercal\in\RR^{nN}$ and a consensus matrix $Z$ corresponding to graph $\cG$,
\$
\min_{x} \sum_{i=1}^n \big[r_i^*(-\Theta_i^\intercal x_i) + \phi^*(x_i)/n\big], \text{s.t. } [(I_{n} - Z)\otimes I_d] x = 0.
\$
\end{proposition}

Proposition \ref{prop:feature} shows the equivalence between Problem \ref{prob:feature} and a form of Problem \ref{prob:constrained}, which is equivalent to Problem \ref{prob:dual}. This suggests that if the gradients (and Hessians) of conjugates $\phi^*$ and $r_i^*$ can be computed efficiently in practice (e.g., by a closed form or polynomial-time algorithms), we can apply DISH to solve the corresponding dual problem in Proposition \ref{prop:feature} instead of the original one in \eqref{prob:feature}. Here 
is an example of when the conjugates can be easily computed.
\begin{example}
Suppose that $\omega\in\RR^{N}$ and $\omega_i\in\RR^{d_i}$, and quadratic functions $\phi(\omega) = \omega^\intercal U\omega/2 + u^\intercal \omega$ and $r_i(\omega_i) = \omega_i^\intercal V_i\omega_i/2 + v_i^\intercal \omega_i$ with $U\in\RR^{N\times N}\succ 0$ and $V_i\in\RR^{d_i\times d_i}\succ 0$ for $i\in\cN$. It is easy to compute the conjugates $\phi^*$ and $r_i^*$ and obtain the dual problem in Proposition \ref{prop:feature} that $\min_x\sum_{i=1}^n[(\Theta_i^\intercal x_i-v_i)^\intercal V_i^{-1}(\Theta_i^\intercal x_i-v_i) + (x_i-u)^\intercal U^{-1}(x_i-u)/n]/2$ such that $[(I_{n} - Z)\otimes I_d] x = 0$. In DISH, we have $P_i^k=Q_i^k=I_N$ for gradient-type updates, and $(P_i^k)^{-1}=Q_i^k =\Theta_iV_i^{-1}\Theta_i^\intercal + U^{-1}/n + \mu I_N$ for Newton-type updates. Thus, when $r_i(\omega_i)$ is the $l_2$ regularizer with $V_i=\chi I_{d_i}$ and $\chi>0$, and the number of samples $N$ is relatively small, we can compute the Newton-type updates efficiently. 
\end{example}

\subsection{Network Flow Optimization Problems}
We now study nonlinear network flow optimization problems over multi-agent networks. We first present the problem setting and its equivalent structured minimax formulation. 
\subsubsection{Problem Formulation}
We recall that agent $i$ locates at node $i$ in the network. In a network flow problem, we define $x\in\RR^{|\cE|}$ as the decision variable with entries $x_{ij}$ for $\{i,j\}\in\cE$. For a convention, 
we ask agent $i$ to control the flow $x_{ij}$ for any $j>i$, and we use  $x_{ij}$  for $i<j$  to denote the directed flow from node $i$ to node $j$. Let $\pi\in\RR^n$ be a given supply vector with entries $\pi_i$ the external supply (demand) when $\pi_i>0$ ($\pi_i<0$) at agent $i$. We assume $\sum_{i\in\cN} \pi_i = 0$ to ensure the total supply equals the total demand over the system. We suppose the cost function is separable in terms of edges with the form $f^{\sf nf}(x) = \sum_{\{i,j\}\in\cE} f_{ij}(x_{ij})$, where $f_{ij}:\RR\to \RR$ is the cost at edge $\{i,j\}$. We study a separable \emph{network flow optimization problem} \cite{bertsekas1989parallel} with the flow balance constraint,
\#\label{prob:nf-original}
    &\min_{x\in \RR^{|\cE|}} \sum_{\{i,j\}\in\cE} f_{ij}(x_{ij}), \\ 
    &\text{ \ s.t. } \sum_{j\colon \{i,j\}\in\cE, j>i} x_{ij} - \sum_{j\colon \{i,j\}\in\cE, j<i} x_{ij} = \pi_i, \ \forall i\in\cN.\notag
\#
By summing up the constraints over all $i\in\cN$, we verify that $\sum_{i\in\cN} \pi_i=0$ as required before. Let $E\in\RR^{n\times |\cE|}$ denote the node-edge incidence matrix with entries $E_{i,\{i,j\}} = 1$ and $E_{j,\{i,j\}} = -1$ if $i<j$ and $E_{k,\{i,j\}} = 0$ if $k\neq i,j$ for $\{i,j\}\in\cE$. We remark that $\ker(E^\intercal)=\text{span}\{\mathds{1}_n\}$ and the Laplacian matrix of $\cG$ can be represented as $EE^\intercal\in\RR^{n\times n}$. For compactness, we rewrite Problem \ref{prob:nf-original} as follows,
\#\label{prob:nf}
\min_{x\in \RR^{|\cE|}} f^{\sf nf}(x), \quad \text{s.t. } Ex=\pi.
\#
Since $\text{im}(E) = \ker(E^\intercal)^{\perp} = \text{span}\{\mathds{1}_n\}^{\perp}$ and $\mathds{1}_n^\intercal \pi=0$, we have $\pi\in\text{im}(E)$. Thus, there exists a feasible $x^{\sf nf}$ to the above problem such that $Ex^{\sf nf}=\pi$.
We impose the following assumption on $f^{\sf nf}$. 
\begin{assumption}\label{ass:nf}
    Function $f^{\sf nf}(x)$ is twice differentiable, $m_{\sf nf}$-strongly convex, and $\ell_{\sf nf}$-Lipschitz smooth with constants $0<m_{\sf nf}\le \ell_{\sf nf}$.
\end{assumption}
Let $y=(y_1;\cdots;y_n)\in\RR^n$ be the dual variable with $y_i$ associated with the constraint $[Ex]_i=\pi_i$ at agent $i$. To solve Problem \ref{prob:nf} with the flow balance constraint, we define the Lagrangian $L^{\sf nf}$ as follows,
\#\label{eq:Lagrangian-nf}
L^{\sf nf}(x,y) = f^{\sf nf}(x) + y^\intercal(Ex-\pi).
\#
By the convexity of $f^{\sf nf}$ and Slater's condition, strong duality holds. Thus, Problem \ref{prob:nf} is equivalent to the following minimax problem, which we refer to as Problem \ref{prob:nf-minimax},
\#\tag{NF}\label{prob:nf-minimax}
\max_{y\in\RR^{n}} \psi^{\sf nf}(y), \text{ where } \psi^{\sf nf}(y) = \min_{x\in\RR^{|\cE|}} L^{\sf nf}(x, y).
\#
As will illustrate after Assumption \ref{ass:Hessian}, given any $y\in\RR^p$, $L^{\sf nf}(\cdot,y)$ is strongly convex with a unique minimizer.

\subsubsection{Distributed Hybrid Methods for Network Flow}
We propose distributed hybrid methods for  solving Problem \ref{prob:nf-minimax}. The hybrid method allows various updating types for primal variables at each iteration. By substituting $\nabla_x L^{\sf nf} = \nabla f^{\sf nf}(x) + E^\intercal y$ and $\nabla_y L^{\sf nf}=Ex - \pi$, the compact form of the distributed hybrid method performs as follows,
\#\label{eq:nf-update}
& x^{k+1} = x^k - AP^k(\nabla f^{\sf nf}(x^k) + E^\intercal y^k), \notag\\ 
& y^{k+1} = y^k + BQ^k(Ex^k - \pi), 
\#
where $A = \diag\{a_{ij}\}\in\RR^{|\cE|\times|\cE|}$ and $B = \diag\{b_i\}\in\RR^{n\times n}$ consist of positive stepsizes $a_{ij}$ for $\{i,j\}\in\cE$ and $b_i$ for $i\in\cN$, and $P^k = \diag\{p_{ij}^k\}\in\RR^{|\cE|\times|\cE|}$ and $Q^k = \diag\{q_i^k\}\in\RR^{n\times n}$ consist of positive scaling values $p_{ij}^k$ for $\{i,j\}\in\cE$ and $q_i^k$ for $i\in\cN$. The scaling values here serve more like personalized stepsizes for each variable and each iteration. Here are examples of possible gradient-type and Newton-type scaling values. Let $\cJ_1^k = \{\{i,j\}\in \cE \colon \text{$x_{ij}$ takes gradient-type updates at iteration $k$}\}$ and $\cJ_2^k = \{\{i,j\}\in \cE \colon \text{$x_{ij}$ takes Newton-type updates at $k$}\}$. We take
\#\label{eq:primal_p_nf}
& \textit{Primal:}\quad \text{Gradient-type, }p_{ij}^k=1; \\
& \qquad \text{Newton-type, } p_{ij}^k = (\nabla^2 f_{ij}(x_{ij}^k))^{-1}, \ \forall \{i,j\}\in\cE.  \notag\\
& \textit{Dual:} \quad  q_i^k = \big[| \{j\colon \{i,j\}\in\cJ_1^k\}| \notag\\
&\qquad \qquad \quad + \sum_{j\colon \{i,j\}\in \cJ_2^k} (\nabla^2 f_{ij}(x_{ij}^k))^{-1} \big]^{-1}, \ \forall i\in\cN. \notag
\#
We will illustrate the scalings in \eqref{eq:primal_p_nf} in the next subsection. Other choices of positive scaling values uniformly bounded over $k$ can also work. Algorithm \ref{alg:nf} shows the distributed implementation of \eqref{eq:nf-update}. It consists of primal (Lines \ref{line:primal-start} - \ref{line:primal-nf}) and dual steps (Line \ref{line:dual-nf}) for each agent. The primal step on $x_{ij}$ using $p_{ij}^k$ in \eqref{eq:primal_p_nf} reflects the flow on edge $\{i,j\}\in\cE$ and is updated using either gradient-type or Newton-type information of its local edge objective $f_{ij}$ along with the two end points' dual variables. The dual step on $y_i$ using $q_i^k$ given by \eqref{eq:primal_p_nf}  corresponds to the flow balance constraint at node $i$ and uses all the primal information from its neighboring edges $x_{ij}$.

\begin{algorithm}[ht] 
    \caption{Distributed Hybrid Method for Network Flow Optimization}
    \begin{algorithmic}[1]\label{alg:nf}
    \STATE{\textbf{Input:} Initialization $x_{ij}^0,y_i^0\in\RR$ and stepsizes $a_{ij},b_i\in \RR^+$ for $\forall i\in \cN$ and $\forall \{i,j\}\in \cE$, respectively.}  
    \FOR{$k =  0, \ldots, K-1$} 
        \FOR{each agent $i\in \cN$ in parallel}
        \STATE{Send values $x_{ij}^{k}$ and $y_i^{k}$ (and $(\nabla^2 f_{ij}(x_{ij}^k))^{-1}$ if $\{i,j\}\in\cJ_2^k$) to $i$'s neighbor $j$;}
        \STATE{Choose its local scale values $p_{ij}^k$ and $q_i^k$;}
            \FOR{each neighbor $j$ ($j$ such that $\{i,j\}\in\cE$) satisfying $j>i$ in parallel}\label{line:primal-start}
           \STATE{$
            x_{ij}^{k+1} = x_{ij}^k - a_{ij}p_{ij}^k(\nabla f_{ij}(x_{ij}^k) + y_i^{k} - y_j^k );$ } \ENDFOR\label{line:primal-nf}
           \STATE{$ y_i^{k+1} =  y_i^k + b_iq_i^k(\sum_{j\colon \{i,j\}\in\cE, j>i} x_{ij}^k - \sum_{j\colon \{i,j\}\in\cE, j<i} x_{ij}^k - \pi_i)$;}\label{line:dual-nf}
        \ENDFOR
    \ENDFOR
\end{algorithmic}
\end{algorithm}

\subsubsection{Special Cases of Algorithm \ref{alg:nf}} \label{sect:nf_nda}
We now illustrate the update choices provided in \eqref{eq:primal_p_nf}. We begin with the two extreme cases with all gradient or Newton-type edges. First, when all edges take gradient-type updates, \eqref{eq:primal_p_nf} implies that $P^k=I_{|\cE|}$ and $BQ^k = \diag\{b_i/\deg(i)\}$ for any $k$ in \eqref{eq:nf-update}. It recovers GDA with personalized stepsizes for $y_i$. 

Next, we consider the case when all edges take Newton-type updates for a speedup.
We have $q_i^k = \sum_{j\colon \{i,j\}\in\cE} (\nabla^2 f_{ij}(x_{ij}^k))^{-1}$ for $i\in\cN$ in \eqref{eq:primal_p_nf} and
\#\label{eq:nf-all-Newton}
P^k=(\nabla^2 f^{\sf nf}(x^k))^{-1} \text{ and } Q^k= \diag\{q_i^k\}.
\# 
We now study both the primal and the dual updates and show that Algorithm \ref{alg:nf} approximates NDA by a diagonalized dual Hessian in this particular case.

\vskip4pt
\noindent\textit{Primal Updates.} The primal Newton's step for solving the inner problem $\min_{x} L^{\sf nf}(x, y)$ in \eqref{prob:nf-minimax} at iteration $k$ is 
\$
x^{k+1} = x^k - \big(\nabla_{xx}^2 L^{\sf nf}(x^k, y^k)\big)^{-1}\nabla_{x} L^{\sf nf}(x^k, y^k).
\$
By substituting $\nabla^2_{xx} L^{\sf nf}(x^k, y^k)= \nabla^2 f^{\sf nf}(x^k)$ and $\nabla_{x} L^{\sf nf}(x^k, y^k)= \nabla f^{\sf nf}(x^k) + E^\intercal y^k$ in NDA, it recovers the primal Newton's step with $P^k$ given in \eqref{eq:nf-all-Newton}.

\vskip4pt
\noindent\textit{Dual Updates.} 
We now consider $y$'s (dual) Newton's update for $\max_{y} \psi^{\sf nf}(y)$ at iteration $k$. We replace $x^*(y^k)$ by the current primal iterate $x^k$ and define $\widehat{\nabla}\psi^{\sf nf}(y^k)$ and $\widehat{\nabla}^2\psi^{\sf nf}(y^k)$ as estimators of ${\nabla}\psi^{\sf nf}(y^k)$ and ${\nabla}^2\psi^{\sf nf}(y^k)$ due to the lack of the exact minimizer $x^*(y^k)$, and obtain 
\$
& \widehat{\nabla}\psi^{\sf nf}(y^k) = Ex^k - \pi, \\
& \widehat{\nabla}^2\psi^{\sf nf}(y^k) = -N^{\sf nf}(x^k,y^k) = -E [\nabla^2 f^{\sf nf}(x^k)]^{-1} E^\intercal.
\$
We remark that $\widehat{\nabla}^2\psi^{\sf nf}(y^k)$ is not full-rank due to the matrix $E$. Let $\Delta y^k$ be dual Newton's update that $y^{k+1} = y^k + \Delta y^k$ defined by $\widehat{\nabla}^2\psi^{\sf nf}(y^k) \Delta y^k = -\widehat{\nabla}\psi^{\sf nf}(y^k)$. Then it satisfies \$E [\nabla^2 f^{\sf nf}(x^k)]^{-1} E^\intercal \Delta y^k = Ex^k - \pi.\$
Since the dual Hessian $E [\nabla^2 f^{\sf nf}(x^k)]^{-1} E^\intercal$ is inseparable, we approximate it by its diagonal part to design a distributed method. We recall that the Laplacian matrix of $\cG$ is $EE^\intercal = D-A^{\sf adj}$, where $D$ is the degree matrix with diagonal entries $D_{ii} = \deg(i)$ for $i\in\cN$ and $0$ otherwise, and $A^{\sf adj}$ is the adjacency matrix with entries $A^{\sf adj}_{ij} = 1$ if $\{i,j\}\in\cE$ and $0$ otherwise. Inspired by this, we split by $E [\nabla^2 f^{\sf nf}(x^k)]^{-1} E^\intercal = D^{\sf nf}(x^k) - A^{\sf nf}(x^k)$, where $D^{\sf nf}(x^k)$ is diagonal with $[D^{\sf nf}(x^k)]_{ii} = \sum_{j\colon \{i,j\}\in\cE} (\nabla^2 f_{ij}(x_{ij}^k))^{-1}$ for $i\in\cN$ and $A^{\sf nf}(x^k)$ has entries $[A^{\sf nf}(x^k)]_{ij} = (\nabla^2 f_{ij}(x_{ij}^k))^{-1}$ if $\{i,j\}\in\cE$ and $0$ otherwise. We approximate $E [\nabla^2 f^{\sf nf}(x^k)]^{-1} E^\intercal$ by its diagonal part $D^{\sf nf}(x^k)$ to obtain a distributed scheme. 
We note that $Q^k = [D^{\sf nf}(x^k)]^{-1}$ in \eqref{eq:nf-all-Newton}. Thus, the dual Newton-type updates with $Q^k$ in \eqref{eq:nf-all-Newton} estimate Newton's steps by adopting the diagonalized Hessian.

We further discuss the updates $P^k$ and $Q^k$ provided in \eqref{eq:primal_p_nf} when the system has both gradient-type and Newton-type edges. The diagonal matrix $P^k$ denotes whether the local update is gradient-type ($p_{ij}^k=1$) or Newton-type ($p_{ij}^k=(\nabla^2 f_{ij}(x_{ij}^k))^{-1}$). Moreover, similar to the Newton-type dual updates, $(Q^k)^{-1}$ in \eqref{eq:primal_p_nf} takes the diagonal part of the matrix $EP^kE^\intercal$ to utilize the primal gradient or Hessian information from adjacent edges as defined in $P^k$.
In summary, Algorithm \ref{alg:nf} 
provides a flexible distributed method when there are both gradient-type and Newton-type edges in the system.

\section{GRAND: Gradient-Related Ascent and Descent Algorithm} \label{sect:alg-GRAND}
Recall that Problems \ref{prob:dual} and \ref{prob:nf-minimax} are in the minimax form of Problem \ref{prob:L}.  Thus, to analyze the performance of our distributed hybrid methods with general update directions, we analyze generalized methods for solving Problem \ref{prob:L}.

\subsection{GRAND} 
We introduce the \emph{\underline{g}radient-\underline{r}elated \underline{a}scent a\underline{n}d \underline{d}escent} (GRAND) algorithmic framework in Algorithm \ref{alg:grand} for solving minimax problems. GRAND presents a generalization of the distributed hybrid methods proposed in Algorithms \ref{alg:dish} and \ref{alg:nf}. In Algorithm \ref{alg:grand}, constants $\alpha$ and $\beta$ are
stepsizes and vectors $s^k$ and $t^k$ are $x$-descent and $y$-ascent update directions, respectively.
\begin{algorithm}
    \caption{GRAND: Gradient-Related Ascent and Descent.}
    \begin{algorithmic}[1]\label{alg:grand}
    \STATE{Input: $\alpha>0$, $\beta>0$, $x^0\in\RR^d$, and $y^0\in\RR^p$. }
    \FOR{$k =  0, \cdots, K-1$} 
    \STATE{Take $s^k\in\RR^d$ and $t^k\in\RR^p$ satisfying Assumption \ref{ass:descent_direct} (or $s^k\in\RR^d$ under Assumption \ref{ass:descent_direct-alt} for Alt-GRAND)}
        \STATE{$x^{k+1} = x^k - \alpha s^k$, }
         \STATE{(Take $t^k\in\RR^p$ under Assumption \ref{ass:descent_direct-alt} for Alt-GRAND)}
        \STATE{$y^{k+1} = y^k + \beta t^k$. }
    \ENDFOR
\end{algorithmic}
\end{algorithm}
GRAND generalizes the gradient descent ascent method (GDA) by allowing updates $s^k$ and $t^k$ to be within uniformly bounded acute angles to the partial gradients. We state the formal assumptions as follows.

\begin{assumption}\label{ass:descent_direct}
There are positive constants $\gamma_s$, $\gamma_t$, $\Gamma_s$, and $\Gamma_t$ such that for any $k$, the updates $s^k$ and $t^k$ satisfy
\$
&\|s^k\| \ge \sqrt{\gamma_s\Gamma_s} \|\nabla_x L(x^k,y^k)\|,\\
& (s^k)^\intercal\nabla_x L(x^k,y^k) \ge \|s^k\|^2/\Gamma_s, \\
&\|t^k\| \ge \sqrt{\gamma_t\Gamma_t} \|\nabla_y L(x^k,y^k)\|,\\
& (t^k)^\intercal\nabla_y L(x^k,y^k) \ge \|t^k\|^2/\Gamma_t.
\$
\end{assumption}
Assumption \ref{ass:descent_direct} is inspired by the gradient-related descent methods for solving minimization problems \cite{bertsekas1989parallel}. For the $x$-update $s^k$, the first condition implies that $s^k\neq 0$ and thus $x^{k+1}\neq x^k$ whenever $\nabla_x L(x^k, y^k) \neq 0$, and the second condition ensures that $-s^k$ is a descent direction with an acute angle to $\nabla_x L(x^k, y^k)$. Similarly, $t^k$ is an ascent direction along $\nabla_y L(x^k, y^k)$. We will provide a general convergence analysis of GRAND in Section \ref{sect:theorem}. GRAND is a general framework that includes some important specific methods. We first note that GDA is a special case of GRAND. 

\vskip4pt
 \emph{GDA.} If we take $s^k = \nabla_x L(x^k, y^k)$ and $t^k = \nabla_y L(x^k, y^k)$ for all $k$, Algorithm \ref{alg:grand} recovers GDA with $\gamma_s = \Gamma_s = \gamma_t = \Gamma_t = 1$ in Assumption \ref{ass:descent_direct}. 

    Besides the gradient method, the gradient-related directions also enable methods adopting scaled gradients, Newton's updates, or quasi-Newton updates. These methods can potentially improve the local numerical performance. 

\vskip4pt
\emph{Scaled Gradient Descent Ascent Method.} Algorithm \ref{alg:grand} leads to the scaled gradient method when $s^k = P^k\nabla_x L(x^k, y^k)$ and $t^k = Q^k\nabla_y L(x^k, y^k)$ with positive definite scaling matrices $P^k$ and $Q^k$. We assume uniformly bounded eigenvalues of $P^k$ and $Q^k$ such that $\sqrt{\gamma_s\Gamma_s} I_d \preceq P^k \preceq \Gamma_s I_d$ and $\sqrt{\gamma_t\Gamma_t} I_p \preceq Q^k \preceq \Gamma_t I_p$ for all $k$ to satisfy Assumption \ref{ass:descent_direct}. 

The scalings $P^k$ and $Q^k$ provide flexibility when designing distributed methods. They can help the system mimic Newton's update and improve numerical performance. In particular, our distributed hybrid methods proposed in Algorithms \ref{alg:dish} and \ref{alg:nf} are special cases of GRAND with scaled gradient updates. 

Moreover, if $P^k=P$ and $Q^k=Q$ are constant matrices, they are also known as the preconditioners. Preconditioners are shown to be crucial in practice when training GANs \cite{wang2019solving}. 

\vskip4pt    
\emph{Newton-type Descent Ascent Method (NDA).}
Algorithm \ref{alg:grand} is a Newton-type method when 
    $s^k = [\nabla_{xx}^2L(x^k,y^k)]^{-1}\nabla_x L(x^k, y^k)$ and $t^k = [N(x^k,y^k)]^{-1}\nabla_y L(x^k, y^k)$. Here $N(x^k,y^k)$ estimates the Hessian $-\nabla^2\psi(y^k)$ by replacing $x^*(y^k)$ with $x^k$. In this method, $x$ takes a Newton's step along $\nabla_x L(x^k, y^k)$ and moves towards $x^*(y^k)$, and $y$ mimics the Newton's step $-[\nabla^2\psi(y^k)]^{-1}\nabla \psi(y^k)$ to maximize $\psi(y)$.
    
    Assumption \ref{ass:descent_direct} holds for $s^k$ with $\Gamma_s = 1/m_{\sf x}$ and $\sqrt{\gamma_s\Gamma_s} = 1/\ell_{\sf xx}$ under Assumption \ref{ass:Hessian}. Moreover, it holds for $t^k$, if there exists a constant $\varrho>0$ such that $N(x,y) \succ \varrho I_p$ for any $(x,y)$. In this case, we have $\Gamma_t = 1/\varrho$ and $\sqrt{\gamma_t\Gamma_t} = 1/(\ell_{\sf yx}\ell_{\sf xy}/m_{\sf x} + \ell_{\sf yy})$. Such a condition is not restrictive. For example, when there is $m_{\sf y}>0$ such that $L(x, y)$ is $m_{\sf y}$-strongly concave with respect to $y$, we have $\nabla^2_{yy}L(x, y)\preceq -m_{\sf y} I_p$. 
    If we further assume the continuity of $\nabla_{yx}^2 L(x,y)$, we have $\nabla_{yx}^2 L(x,y) = (\nabla_{xy}^2 L(x,y))^\intercal$ by Clairaut's theorem. Thus, we have $\nabla_{yx}^2 L(x,y) [\nabla_{xx}^2L(x,y)]^{-1}\nabla_{xy}^2L(x,y) \succeq 0_p$ since $[\nabla_{xx}^2L(x,y)]^{-1} \succeq 1/\ell_{\sf xx} \cdot I_d$. In this case, we can take $\varrho = m_{\sf y}/2$ and Assumption \ref{ass:descent_direct} holds for $N$ defined in \eqref{eq:N-operator}. 

\vskip4pt  
    \emph{Quasi-Newton-type Descent Ascent Method.} The aforementioned scalings $P^k$ and $Q^k$ can also be quasi-Newton updates, like (L)-BFGS matrices. For example, PD-QN \cite{eisen2019primal}, the distributed primal-dual quasi-Newton method for consensus problems is a special case of GRAND under Assumption \ref{ass:descent_direct}.

\subsection{Alternating GRAND} \label{sect:grand_alt}
We now introduce Alt-GRAND as an alternating version of GRAND, where the updates for $x$ and $y$ are performed sequentially (Gauss-Seidel updates \cite{bertsekas1989parallel}) instead of simultaneously (Jacobi updates). Alt-GRAND adopts the updates in Algorithm \ref{alg:grand} with a different assumption that the $y$-update $t^k$ is along the alternating partial gradient using the updated $x^{k+1}$.
Formally, we present the following assumption. 
\begin{assumption}[Alt-GRAND]\label{ass:descent_direct-alt} 
There are positive constants $\gamma_s$, $\gamma_{\tau}$, $\Gamma_{s}$ and $\Gamma_{\tau}$ such that $s^k$ and $t^k$ in Algorithm \ref{alg:grand} satisfy
$
\|s^k\| \ge \sqrt{\gamma_s\Gamma_s} \|\nabla_x L(x^k,y^k)\|$, $(s^k)^\intercal\nabla_x L(x^k,y^k) \ge \|s^k\|^2/\Gamma_s$, 
$\|t^k\| \ge \sqrt{\gamma_{\tau}\Gamma_{\tau}} \|\nabla_y L(x^{k+1},y^k)\|$,  and $(t^k)^\intercal\nabla_y L(x^{k+1},y^k) \ge \|t^k\|^2/\Gamma_{\tau}$.
\end{assumption}
In Alt-GRAND, $s^k$ satisfies the same conditions as in GRAND, while $t^k$ is an ascent direction along the updated $\nabla_y L(x^{k+1}, y^k)$ instead of $\nabla_y L(x^k, y^k)$. Alt-GRAND is a generalization of Alt-GDA, which has been shown to outperform GDA numerically in some cases \cite{bailey2020finite}. We analyze its convergence in Section \ref{sect:thm-Alt-GRAND}, and compare its numerical performance with GRAND in Section \ref{sect:exp-centralized}.

Alt-GRAND allows scaled implementations if $s^k = P^k\nabla_x L(x^k, y^k)$ and $t^k = Q^k\nabla_y L(x^{k+1}, y^k)$ with positive definite matrices $P^k$ and $Q^k$ satisfying $\sqrt{\gamma_s\Gamma_s} I_d \preceq P^k \preceq \Gamma_s I_d$ and $\sqrt{\gamma_\tau\Gamma_\tau} I_p \preceq Q^k \preceq \Gamma_\tau I_p$. Newton-type methods are also covered by Alt-GRAND. For example, GDN \cite{zhang2020newton} is a special case when $P^k=\nabla_{xx}^2 L(x^k, y^k)$ and $Q^k = I_p$, which has a provable local linear rate. The 
alternating Newton-type method (Alt-NDA), on the other hand, takes $P^k = [\nabla_{xx}^2L(x^k,y^k)]^{-1}$ and $Q^k = [N(x^{k+1},y^k)]^{-1}$, similar to NDA. Assumption \ref{ass:descent_direct-alt} holds for $t^k$ under similar conditions as in NDA. Alt-NDA, also known as the complete Newton method \cite{zhang2020newton}, has a provable local quadratic rate. We will discuss further the local performance of Alt-NDA in Section \ref{sect:local-AltNDA}.

\section{Global Convergence Analysis} \label{sect:global}
In this section, we analyze the global convergence of GRAND. Theorem \ref{thm:PL} establishes the linear convergence of GRAND under certain strongly-convex-PL conditions, which ensures the linear rate of the distributed hybrid methods.

\subsection{Preliminaries} 
We first introduce assumptions and definitions used throughout the section, starting with the standard conditions for $L$.
\begin{assumption}
\label{ass:Hessian}
The function $L(x,y)$ satisfies that,
\begin{enumerate}
    \item[(a)] $L$ is twice differentiable in
    $(x, y)$. Its partial gradient $\nabla_x L$ is continuously differentiable relative to $(x,y)$; 
    \item[(b)] Given any $y\in\RR^p$, $L(\cdot, y)$ is $m_{\sf x}$-strongly convex with respect to $x$ with $m_{\sf x}>0$;
    \item[(c)] The partial gradient  $\nabla_x L$ is $\ell_{\sf xx}$- and $\ell_{\sf xy}$-Lipschitz continuous in $x$ and $y$, respectively. Moreover, $\nabla_y L$ is $\ell_{\sf yx}$- and $\ell_{\sf yy}$-Lipschitz continuous in $x$ and $y$, respectively. Here, constants $\ell_{\sf xx}>m_{\sf x}>0$, and $\ell_{\sf xy},\ell_{\sf yx},\ell_{\sf yy}\ge0$.
\end{enumerate}
\end{assumption}
It is easy to check that $L^{\sf dc}$ defined in \eqref{eq:lagrangian_func} under Assumption \ref{ass:Hessian-dist} satisfies Assumption \ref{ass:Hessian} with $m_{\sf x}=m_{\sf dc}$, $\ell_{\sf xx} = \ell_{\sf dc} + 2\mu$, $\ell_{\sf xy} = \ell_{\sf yx} =2$, and $\ell_{\sf yy} = 0$. Moreover, $L^{\sf nf}$ defined in \eqref{eq:Lagrangian-nf} under Assumption \ref{ass:nf} satisfies Assumption \ref{ass:Hessian} with $m_{\sf x}=m_{\sf nf}$, $\ell_{\sf xx} = \ell_{\sf nf}$, $\ell_{\sf xy} = \ell_{\sf yx} = \|E\|$, and $\ell_{\sf yy} = 0$. 

Most existing analyses of GDA in strongly-convex-concave settings study linear combinations of $\|x^k-x^\star\|^2$ and $\|y^k-y^\star\|^2$ as Lyapunov functions \cite{du2019linear}, where $(x^\star, y^\star)$ is a solution to Problem \ref{prob:L}. Let $z^\star = (x^\star; y^\star)$ and $z^k = (x^k; y^k)$. 
The gradient steps in GDA decrease the Lyapunov function by a ratio such that ~$\|z^{k+1}-z^\star\|_{V}^2\le \rho\|z^k-z^\star\|_{V}^2$ for a matrix $V\succeq 0$ and a constant $0<\rho<1$ at iteration $k$, implying a linear rate. However, such analysis does not apply to GRAND due to the time-varying angles between the updates and the gradients. 
A similar procedure leads to ~$\|z^{k+1}-z^\star\|_{V^k}^2\le \rho\|z^k-z^\star\|_{V^k}^2$ with time-varying matrices $\{V^k\succeq 0\}_k$, which does not ensure convergence. Moreover, such Lyapunov functions also fail in nonconcave 
cases. Thus, recalling $x^*(y)$ and $\psi(y)$ defined in \eqref{eq:x_star_y}, we introduce two performance metrics, $y$'s optimality measure and $x$'s tracking error, with $\Xi_\psi = \max_y \psi(y)$,
\#\label{eq:terrors}
& \Delta_{y}^k = \Xi_\psi - \psi(y^k), \notag\\
& \Delta_{x}^k = L(x^k, y^k) - L(x^*(y^k),y^k).
\#
We remark that $\Delta_{x}^k$ and $\Delta_{y}^k$ are nonnegative by definition. Here $\Delta_{y}^k$ measures the distance between $y$'s current function value to 
its upper bound, and $\Delta_{x}^k$ tracks the error of $x$'s current function value to the optimal one at the current $y^k$ point.  We take $\Xi_\psi=\psi(y^\star)$ when $\psi$ has a maximizer $y^\star$. In this case, $\Delta_y^k$ is $y$'s optimality gap and becomes zero at the optimal point $(x^\star, y^\star)=(x^*(y^\star), y^\star)$. We will define Lyapunov functions as linear combinations of these performance metrics. 

\subsection{Global Convergence of GRAND} \label{sect:theorem}
This section analyzes the global convergence of GRAND under Assumption \ref{ass:descent_direct}. We first define some constants used in the analysis. When $\gamma_t<\Gamma_t$, let $\nu=1/\sqrt[3]{1-\gamma_t^2/\Gamma_t^2} - 1>0$. 
We define positive constants $c_1$, $\ell_\psi$, $\iota$, and $c_2$ as follows,
\#\label{eq:constants_c}
& c_1 = (\Gamma_t^2/2\gamma_t)[{\nu}\II_{\{\gamma_t<\Gamma_t\}}/(1+\nu) + \II_{\{\gamma_t=\Gamma_t\}}], \\
& \ell_\psi = \ell_{\sf yy} + {\ell_{\sf yx}\ell_{\sf xy}}/{m_{\sf x}}, \quad
 \iota =  2\Gamma_t + \Gamma_t^2 + {c_1\ell_{\sf yy}}/(3\ell_\psi), \notag\\
 & c_2 = (\Gamma_t^2/2\gamma_t)\{[1/(1+\nu) + 1+ \nu]\II_{\{\gamma_t<\Gamma_t\}}/\nu + \II_{\{\gamma_t=\Gamma_t\}}\}. \notag
\#
Then we define functions $\Upsilon^k$ and $\Delta^k$ as combinations of $\Delta_y^k$ and $\Delta_x^k$ defined in \eqref{eq:terrors}. For $k=0,1,\cdots,K$, we have
\#\label{eq:Delta_Upsilon}
& \Upsilon^k = \beta\iota\|\nabla \psi(y^{k})\|^2 + (2\alpha\gamma_sm_{\sf x}/3)\Delta_{x}^k, \notag \\
& \Delta^k = (3\iota/c_1)\Delta_y^{k} + \Delta_x^{k},
\# 
We remark that $\Upsilon^k$ and $\Delta^k$ are nonnegative. The function $\Upsilon^k$ is the Lyapunov function measuring the performance of GRAND for strongly-convex-nonconcave problems, while $\Delta^k$ is the Lyapunov function in the strongly-convex-PL setting.

\subsubsection{Strongly-Convex-Nonconcave Settings}
We present the sublinear convergence of GRAND in the following theorem.
\begin{theorem}[Strongly-Convex-Nonconcave] \label{thm:nonconvex}
Under Assumptions \ref{ass:descent_direct} and \ref{ass:Hessian}, with constants $c_1$, $c_2$, $\iota$, and $\ell_\psi$ defined in \eqref{eq:constants_c}, suppose the stepsizes satisfy
\$
\alpha \le {2\gamma_s}/\{\Gamma_s^2[3\ell_{\sf xx} + c_1\ell_{\sf yx}^2/(\ell_\psi\Gamma_t^2)]\}, \quad \beta\le \min\{{c_1}/{3\ell_\psi\Gamma_t^2}, {\alpha \gamma_s m_{\sf x}^2c_1}/{[3\ell_{\sf yx}^2\iota(3c_2 + 2c_1)]}\}.\$
Then with $\Upsilon^k$, $\Delta^0$, and $\Delta^K$ defined in \eqref{eq:Delta_Upsilon}, the iterates from Algorithm \ref{alg:grand} satisfy 
\$
\big(\sum_{k=0}^{K-1}\Upsilon^k\big)/K \le \Delta^0/K.
\$
\end{theorem}

\begin{proof}[Proof Sketch of Theorem \ref{thm:nonconvex}]
Since $x$ and $y$-updates are coupled in the descent ascent framework, our idea is to bound $y$'s optimality measure $\Delta_{y}^{k}$ and $x$'s tracking error $\Delta_{x}^k$ through coupled inequalities. We decompose our analysis into four steps. {\it Step 1:} Preparation. {\it Step 2:} We bound $y$'s optimality measure $\Delta_{y}^{k+1}$ with $x$'s tracking error measured in $\|\nabla_{x}L(x^k,y^k)\|^2$ by the Lipschitz continuity of $\nabla \psi$, see Proposition \ref{prop:gL}. {\it Step 3:} We bound $x$'s tracking error $\Delta_{x}^{k+1}$ with $y$'s optimality measure $\|\nabla \psi(y^k)\|$ by the Lipschitz continuity of $\nabla_x L(x,y)$ and $\nabla_y L(x,y)$, see Proposition \ref{prop:Lg}. {\it Step 4:} Finally, we take a linear combination of the coupled bounds on $\Delta_{y}^{k+1}$ and $\Delta_{x}^{k+1}$. We show more details in the sequel.

\vskip4pt
{\it Step 1: Preparation.}
We present some basic lemmas for the analysis. We start with the Lipschitz continuity of $\nabla \psi(y)$, derived based on the Lipschitz continuity of $x^*(y)$. Similar results are also shown in \cite{nesterov2005smooth,rockafellar2009variational}.  
\begin{lemma}\label{lem:psi_property}
Under Assumption \ref{ass:Hessian}, $\nabla\psi(y)$ is $\ell_\psi$-Lipschitz continuous with $\ell_\psi$ defined in \eqref{eq:constants_c}. 
\end{lemma}
The next lemma restates properties of update directions $s^t$ and $t^k$ under Assumption \ref{ass:descent_direct}.
\begin{lemma}\label{lem:descent_direct}
Under Assumption \ref{ass:descent_direct}, for any $k$, it holds that $\gamma_s\le\Gamma_s$ and $\gamma_t\le\Gamma_t$. Moreover,
\$
&\gamma_s \|\nabla_x L(x^k,y^k)\|^2 \le (s^k)^\intercal\nabla_x L(x^k,y^k),\quad \|s^k\|\le \Gamma_s\|\nabla_x L(x^k,y^k)\|, \\
&\gamma_t \|\nabla_y L(x^k,y^k)\|^2 \le (t^k)^\intercal\nabla_y L(x^k,y^k),\quad \|t^k\|\le \Gamma_t\|\nabla_y L(x^k,y^k)\|.
\$
\end{lemma}
Now we show a corollary of the $m_{\sf x}$-strong convexity of $L(x, y)$ with respect to $x$.
\begin{lemma} \label{lem:wx_sc}
Under Assumption \ref{ass:Hessian}, for any $k$, the iterates from Algorithm \ref{alg:grand} satisfy
\$
\|\nabla_y L(x^k,y^k) - \nabla \psi(y^k)\| \le (\ell_{\sf yx}/m_{\sf x})\|\nabla_xL(x^k, y^k)\|.
\$
\end{lemma}
The following lemma provides an upper bound on the $y$-update.  

\begin{lemma}\label{lem:de_lambda}
Under Assumption \ref{ass:Hessian}, for any $\upsilon>0$ and any $k$, the iterates from Algorithm \ref{alg:grand} satisfy
\$
\|\nabla_y L(x^k,y^k)\|^2
\le (1+\upsilon) \|\nabla \psi(y^k)\|^2 + [(1+1/\upsilon)\ell_{\sf yx}^2/m_{\sf x}^2]\|\nabla_xL(x^k,y^k)\|^2.
\$
Further with Assumption \ref{ass:descent_direct}, it holds for any $k = 0,1,\cdots,K-1$ that
\$
\|y^{k+1} - y^k\|^2\le 2\beta^2\Gamma_t^2 \|\nabla \psi(y^k)\|^2 + (2\beta^2\Gamma_t^2\ell_{\sf yx}^2/m_{\sf x}^2)\|\nabla_xL(x^k,y^k)\|^2.
\$
\end{lemma}

\vskip4pt
{\it Step 2: Bounding $y$'s Optimality Measure $\Delta_{y}^{k+1}$.} 
 We first bound $y$'s updated optimality measure $\Delta_{y}^{k+1}$ with $x$'s tracking error measured in $\|\nabla_{x}L(x^k,y^k)\|^2$. The analysis follows from the $\ell_\psi$-Lipschitz continuity of $\nabla \psi$ in Lemma \ref{lem:psi_property}. 
The following proposition shows the obtained upper bound on $\Delta_{y}^{k+1}$.

\begin{proposition} \label{prop:gL}
Under Assumptions \ref{ass:descent_direct} and \ref{ass:Hessian}, with constants $c_1$, $c_2$, and $\ell_\psi$ defined in \eqref{eq:constants_c}, for all $k = 0,1,\cdots,K-1$, the iterates from Algorithm \ref{alg:grand} satisfy
$
\Delta_{y}^{k+1} \le \Delta_{y}^{k} -(c_1-\beta \ell_\psi\Gamma_t^2)\beta\|\nabla \psi(y^{k})\|^2 + [(c_2+\beta\ell_\psi\Gamma_t^2)\beta\ell_{\sf yx}^2/m_{\sf x}^2]\|\nabla_xL({x}^k,y^k)\|^2$.
\end{proposition}
We remark that by checking the derivatives, constants $c_1$ and $c_2$ are monotonically increasing and decreasing relative to the ratio $\gamma_t/\Gamma_t$, respectively. It implies that if the ascent direction $t^k$ lies in a smaller angle to $\nabla_yL(x^k,y^k)$ (as $\gamma_t$ gets closer to $\Gamma_t$), $c_1$ gets larger while $c_2$ gets smaller. In the extreme case when $t^k=\nabla_yL(x^k,y^k)$ ($\gamma_t = \Gamma_t=1$), we have $c_1=c_2=1/2$. This provides the tightest upper bound on $\Delta_y^{k+1}$ in Proposition \ref{prop:gL} compared to other update directions.

\vskip4pt
{\it Step 3: Bounding $x$'s Tracking Error $\Delta_{x}^{k+1}$.}  
Next, we bound $x$'s updated tracking error $\Delta_{x}^{k+1}$ with $y$'s optimality measure $\|\nabla \psi(y^k)\|$. We define constants $\iota_1 = \Gamma_t(2+\Gamma_t + \beta\ell_{\sf yy}\Gamma_t)$ and $\iota_2 = \alpha[\gamma_s-\alpha\Gamma_s^2(\ell_{\sf xx} + \beta\ell_{\sf yx}^2)/2] - \beta\iota_1\ell_{\sf yx}^2/m_{\sf x}^2$.
The conditions of $\alpha$ and $\beta$ in Theorem \ref{thm:nonconvex} ensures $\iota_2>0$. The Lipschitz continuity of $\nabla_x L(x,y)$ in $x$ and $\nabla_y L(x,y)$ in $x$ and $y$ gives the following result.
\begin{proposition} \label{prop:Lg}
Under Assumptions \ref{ass:descent_direct} and \ref{ass:Hessian}, with constants $\iota_1$ and $\iota_2$ defined above, for all $k = 0,1,\cdots,K-1$, the iterates from Algorithm \ref{alg:grand} satisfy
$\Delta_{{x}}^{k+1} \le \Delta_{{x}}^k  + \beta\iota_1\|\nabla \psi(y^k)\|^2 - \iota_2\|\nabla_x L(x^k,y^k)\|^2 + \Delta_{y}^k - \Delta_{y}^{k+1}$.
\end{proposition}

{\it Step 4: Putting Things Together.}
We take a linear combination of the coupled inequalities in Propositions \ref{prop:gL} and \ref{prop:Lg} and obtain the following result. 
\begin{proposition}\label{prop:pre}
Under Assumptions \ref{ass:Hessian} and \ref{ass:descent_direct}, suppose that the stepsizes satisfy the conditions in Theorem \ref{thm:nonconvex}.
Then for all $k = 0,1,\cdots, K-1$, the iterates from Algorithm \ref{alg:grand} satisfy 
\$
(3\iota/c_1+1)\Delta_{y}^{k+1} + \Delta_{x}^{k+1} &\le  (3\iota/c_1+1)\Delta_{y}^{k} - \beta\iota \|\nabla \psi(y^{k})\|^2 + (1- {2\alpha\gamma_sm_{\sf x}}/{3})\Delta_{x}^k,
\$
where constants $c_1$ and $\iota$ are defined in \eqref{eq:constants_c}.
\end{proposition}
Finally, we conclude the proof of Theorem \ref{thm:nonconvex} by substituting $\Upsilon^k$ and $\Delta^k$. 
\end{proof}

\begin{remark}
Due to space constraints, we omit the proof here. Interested readers are referred to \cite{niu2022grand}.
Theorem \ref{thm:nonconvex} presents the global sublinear convergence of GRAND in strongly-convex-nonconcave settings. To illustrate the result, let $\ell=\ell_{\sf xx} + \ell_{\sf xy} + \ell_{xy} + \ell_{\sf yy}$ and $\kappa = \ell/m_{\sf x}$. Here $\ell$ and $\kappa$ characterize the Lipschitz continuity and the condition number of $L$, respectively. We note that $\iota=O(1)$, $\alpha=O(1/\ell)$, and $\beta=O(1/(\kappa^2\ell))$ under the conditions in Theorem \ref{thm:nonconvex}. Theorem \ref{thm:nonconvex} implies that $ (\sum_{k=0}^{K-1}\|\nabla \psi(y^{k})\|^2)/K \le (\sum_{k=0}^{K-1}\Upsilon^k)/(\beta \iota K) \le \Delta^0/(\beta \iota K)$ by the definition of $\Upsilon^k$.
Thus, we need $K=O(\kappa^2\epsilon^{-2})$ iterations to achieve $\min_{k=0,\cdots,K-1}\{\|\nabla \psi(y^{k})\|\}\le \epsilon$. Our iteration complexity and stepsizes all match the state-of-the-art rate for GDA in the same setting \cite{lin2020gradient}. As $\{\Upsilon^k\}_{k\ge 0}$ goes to zero, both $\{\|\nabla \psi(y)\|\}_{k\ge 0}$ and $\{\Delta_x^k\}_{k\ge 0}$ goes to zero, which implies that the iterates converge to a point $(x^*(y^\dagger), y^\dagger)$ with $\nabla \psi(y^\dagger)=0$. Convergence to a stationary point in $y$ is the best we can obtain for strongly-convex-nonconcave problems.

In general, the theoretical convergence speed of the scaled gradient methods has worse constants  than the gradient methods since $\Gamma_s/\gamma_s$ and $\Gamma_t/\gamma_t$ used in the directions are larger than $\Gamma_s/\gamma_s=\Gamma_t/\gamma_t=1$ used in gradient methods. But these scaling methods under GRAND can provide not only more flexibility but also faster convergence behaviors in practice. See Section \ref{sect:exp} for more numerical studies and details.
\end{remark}

\subsubsection{Linear Rates for Strongly-Convex-PL Settings}
The preceding result can be strengthened to a linear rate if we further impose the assumption that $\psi$ satisfies the following Polyak-\L{}ojasiewicz (PL) inequality.
\begin{assumption}\label{ass:psi_PL}
For any $y\in\RR^p$, the function $\psi$ defined in \eqref{eq:x_star_y} has a global maximizer and $-\psi$ satisfies the PL inequality with a positive constant $p_\psi$.
\end{assumption} 
Let $\psi^\star$ be the maximum function value. Assumption \ref{ass:psi_PL} gives that for any $y$, $
\|\nabla \psi(y)\|^2/2 \ge p_\psi(\psi^\star - \psi(y))$. 
PL inequality is a simple sufficient condition to show a global linear rate for gradient descent method on solving minimization problems \cite{karimi2016linear}. As an example, we next show that Assumption \ref{ass:psi_PL} can be easily satisfied by distributed computing problems. We introduce a structured problem with $f:\RR^d\to\RR$, $g:\RR^p\to\RR$, and $W\in\RR^{p\times d}$ as follows,
\#\label{eq:structured_minimax}
L(x,y) = f(x) + y^\intercal Wx - g(y).
\#
\begin{example}\label{ex:g-h}
    In a structured problem of the form \eqref{eq:structured_minimax}, if there exists a function $h:\RR^d \to\RR$ such that $g(y)=h(W^\intercal y)$ and $h(\lambda) + m_h\|\lambda\|^2/2$ is convex with $m_h < 1/\ell_{\sf xx}$, then  Assumption \ref{ass:psi_PL} holds with $p_\psi = \sigma_{\min}^+(W)(1/\ell_{\sf xx} - m_h)$.
\end{example}

 For Problem \ref{prob:dual} with $L^{\sf dc}$ defined in \eqref{eq:lagrangian_func}, Example \ref{ex:g-h} holds with $h=g=0$ and $m_h=0$. Thus, with $\sigma_{\min}^+(W) = 1-\gamma$ and $\ell_{\sf xx} = \ell_{\sf dc} + 2\mu$, we have $p_{\psi}^{\sf dc} = (1-\gamma)/(\ell_{\sf dc} + 2\mu)$. Similarly, for Problem \ref{prob:nf-minimax} with $L^{\sf nf}$ defined in \eqref{eq:Lagrangian-nf}, we have $\ell_{\sf xx} = \ell_{\sf nf}$. Since there exists a feasible solution $x^{\sf nf}$ such that $Ex^{\sf nf}=\pi$, Example \ref{ex:g-h} holds with $g(y)=\pi^\intercal y = (x^{\sf nf})^{\intercal} E^\intercal y$ and $h(\lambda)=(x^{\sf nf})^{\intercal}\lambda$ and thus $m_h=0$. Thus, we obtain $p_{\psi}^{\sf nf} = \sigma_{\min}^+(E)/\ell_{\sf nf}$. In addition to the distributed computing problems, Assumption \ref{ass:psi_PL} can also be naturally satisfied by various cases. For example, Lagrangian functions corresponding to any feasible linearly constrained minimization problems are covered by Example \ref{ex:g-h} and thus satisfy Assumption \ref{ass:psi_PL}. Further examples are presented below.

\begin{example}\label{ex:side-PL}
    If $L$ satisfies the one-sided PL condition \cite{yang2020global} with respect to $y$ with $p_{\sf y}>0$, that is, 
    $
    \|\nabla_y L(x,y)\|^2 \ge 2p_{\sf y}[\max_y L(x,y) - L(x,y)]
    $ for any $x, y$, then Assumption \ref{ass:psi_PL} holds with $p_\psi=p_{\sf y}$.
\end{example}
A particular case of Example \ref{ex:side-PL} is when $L$ is $p_{\sf y}$-strongly concave with respect to $y$ with $p_{\sf y}>0$ for any $x\in\RR^d$. Thus, Assumption \ref{ass:psi_PL} holds for strongly-convex-strongly-concave problems.

\begin{example}\label{ex:full-rank}
    In a structured problem of the form \eqref{eq:structured_minimax}, if $W$ has full row-rank and the function $g(y) + m_g\|y|^2/2$ is convex with a constant $m_g<\sigma_{\min}^2(W)/\ell_{\sf xx}$, then it holds that $\sigma_{\min}(W)>0$ and Assumption \ref{ass:psi_PL} holds with $p_\psi = \sigma_{\min}^2(W)/\ell_{\sf xx}-m_g$.
\end{example}
Example \ref{ex:full-rank} holds for many applications including reinforcement learning, empirical risk minimization, and robust optimization problems \cite{du2019linear}. The authors of \cite{du2019linear} consider the case when $g$ is convex ($m_g=0$) and show a linear convergence of GDA.

Now we study the global convergence of GRAND for strongly-convex-PL problems under Assumption \ref{ass:psi_PL}. In particular, we take $\Xi_\psi=\psi(y^\star)$ as the exact upper bound and have $\Delta_y^k = \psi(y^\star) - \psi(y^k)$ in \eqref{eq:terrors}. We note that $\Delta_y^k=0$ and $\Delta_{x}^k=0$ and thus $\Delta^k$ defined in \eqref{eq:Delta_Upsilon} is zero at a global minimax point $(x^k, y^k) = (x^*(y^\star), y^\star)$. For convenience, we define a positive constant $\delta$ with $c_1$ and $\iota$ defined in \eqref{eq:constants_c},
\#\label{eq:delta}
\delta = \min\{{2\beta\iota p_\psi c_1}/(3\iota + c_1), {2\alpha\gamma_sm_{\sf x}}/{3}\}.
\#
The following theorem states the result, where $\delta$ serves as the linear rate coefficient.
\begin{theorem}[Strongly-Convex-PL] \label{thm:PL}
Under Assumptions \ref{ass:descent_direct}, \ref{ass:Hessian}, and \ref{ass:psi_PL}, suppose the stepsizes satisfy the conditions in Theorem \ref{thm:nonconvex} and additionally, $\beta < (3\iota + c_1)/(2\iota p_\psi c_1)$.
For all $k = 0,1,\cdots, K-1$, the iterates from GRAND satisfy 
\$
\Delta^{k+1} \le (1-\delta)\Delta^k,
\$
where $\delta$ is defined in \eqref{eq:delta} satisfying $0<\delta<1$.
\end{theorem}

\begin{remark}
    Theorem \ref{thm:PL} presents the global linear (Q-linear) rate of GRAND for strongly-convex-PL problems under Assumption \ref{ass:psi_PL}. As $\{\Delta^k\}$ goes to zero, both $\{\Delta_y^k\}$ and $\{\Delta_x^k\}$ goes to zero. Moreover, if $y^\dagger$ is a unique maximzer of $\psi$, the theorem ensures the convergence of the iterates to the global minimax point $(x^*(y^\dagger), y^\dagger)$.
    
    We recall that the distributed consensus and the network flow problems mentioned in Example \ref{ex:g-h} satisfy Assumption \ref{ass:psi_PL}. Thus, Theorem \ref{thm:PL} guarantees the global linear convergence of DISH in Algorithm \ref{alg:dish} for solving Problem \ref{prob:dual} and Algorithm \ref{alg:nf} for solving Problem \ref{prob:nf-minimax}. A specialized linear result with tighter coefficients for DISH is presented in \cite{niu2022dish}. 
    
    We define $\widetilde\kappa = \ell/\min\{m_{\sf x}, p_{\psi}\}$ to characterize the condition number of $L$ under Assumptions \ref{ass:Hessian} and \ref{ass:psi_PL}. We investigate the rate coefficient $1-\delta$ by substituting upper bounds on $\alpha$ and $\beta$ and obtain $\delta=O(1/\widetilde\kappa^3)$. It is slower than the optimal complexity $O(1/\widetilde\kappa^2)$ of GDA in strongly-convex-strongly-concave case \cite{nesterov2011solving} since we study a more general strongly-convex-PL setting here.
    In short, $\delta$ depends on the function property $\widetilde\kappa$ and update angles $\Gamma_t/\gamma_t$ and $\Gamma_s/\gamma_s$. Though the theorem is conservative, relying on the worst case of update direction angles,
as experiments show in Section \ref{sect:exp}, scaling matrices and Newton-type updates can accelerate the numerical performance.
\end{remark}

\subsection{Global Convergence of Alt-GRAND}\label{sect:thm-Alt-GRAND}

We now present global rates of Alt-GRAND. The analysis follows the same steps as those for GRAND in Section \ref{sect:theorem}. Let positive constants $\widetilde\iota = 2\Gamma_\tau + {c_1\ell_{\sf yy}}/(3\ell_\psi)$ and $\widetilde\delta = \min\{{2\beta\widetilde\iota p_\psi c_1}/(3\widetilde\iota + c_1), {2\alpha\gamma_sm_{\sf x}}/{3}\}$, and Lyapunov functions $\widetilde\Upsilon^k= \beta\widetilde\iota\|\nabla \psi(y^{k})\|^2 + 2\alpha\gamma_sm_{\sf x}\Delta_{x}^k/3$ and $\widetilde\Delta^k= (3\widetilde\iota/c_1)\Delta_{y}^{k} + \Delta_{ x}^{k}$ be linear combinations of $\Delta_x^k$ and $\Delta_y^k$.  
The convergence results are shown as follows.
\begin{theorem}[Strongly-Convex-Nonconcave] \label{thm:nonconvex-alt}
We assume the stepsizes to satisfy some conditions that $\alpha=O(1/\ell)$ and $\beta=O(1/(\kappa^2\ell))$. Under Assumptions \ref{ass:Hessian} and \ref{ass:descent_direct-alt}, the iterates from Alt-GRAND satisfy 
$(\sum_{k=0}^{K-1}\widetilde\Upsilon^k)/K \le {\widetilde\Delta^0}/{K}$. 
\end{theorem}
\begin{theorem}[Strongly-Convex-PL] \label{thm:PL-alt} Suppose that the stepsizes satisfy some conditions that $\alpha=O(1/\ell)$ and $\beta=O(1/(\kappa^2\ell))$.
It holds that $0<\widetilde\delta<1$. Moreover, under Assumptions \ref{ass:Hessian}, \ref{ass:descent_direct-alt}, and \ref{ass:psi_PL}, for all $k = 0,1,\cdots, K-1$, the iterates from Alt-GRAND satisfy $
\widetilde\Delta^{k+1} \le (1-\widetilde\delta)\widetilde\Delta^k$.
\end{theorem}

Theorems \ref{thm:nonconvex-alt} and \ref{thm:PL-alt} demonstrate the global sublinear and linear convergence rates of Alt-GRAND for strongly-convex-nonconcave and strongly-convex-PL scenarios, respectively. These results are similar to the rates achieved by GRAND in Theorem \ref{thm:PL}, with only differences in the coefficient constants. Though comparisons between the theoretical results may not be straightforward, we will evaluate their numerical performance later. 
Besides a global rate guarantee, Alt-NDA, the Newton-type method, exhibits local quadratic convergence \cite{zhang2020newton} when $\alpha=\beta=1$. Further discussion on Newton-based methods and their local higher-order rates will be presented in the following section.

\section{Newton-Based Methods and Local Higher-Order Rates} \label{sect:local-Newton}
Though we do not obtain superlinear rates for distributed hybrid methods due to distributed approximation errors, this section discusses related Newton-based methods for solving Problem \ref{prob:L} in centralized settings, where an exact Newton-based step is feasible. We study the local quadratic rates of Alt-NDA and its variants in Section \ref{sect:local-AltNDA}. We also explore a modified Newton's method in Section \ref{sect:local-cubic} to achieve local cubic rates by reusing the Hessian inversion computation.  

\subsection{Multistep Alt-NDA with Local Quadratic Rates} \label{sect:local-AltNDA}
We define two mappings $X,Y:\RR^{d}\times\RR^{p}\to\RR^{d}\times\RR^{p}$ using the operator $N$ defined in \eqref{eq:N-operator},
\$
& X(x,y) = (x - [\nabla_{xx}^2 L(x,y)]^{-1}\nabla_xL(x,y), y ), \\
& Y(x,y) = (x, y + [N(x, y)]^{-1}\nabla_yL(x,y)).
\$ 
Throughout this section, we consider $(x^\dagger,y^\dagger)$ as a first-order stationary point of Problem \ref{prob:L}. We focus on the local performance and assume  $[N(x,y)]^{-1}$ is well-defined (not necessarily semi-definite) in a neighborhood around $(x^\dagger,y^\dagger)$. Specifically, we consider Alt-NDA, introduced in Section \ref{sect:grand_alt}, with $s^k = [\nabla_{xx}^2 L(x^k,y^k)]^{-1}\nabla_xL(x^k,y^k)$ and $t^k = [N(x^{k+1}, y^k)]^{-1}\nabla_yL(x^{k+1},y^k)$. Alt-NDA can be represented as the following composite update,
\$
(x^{k+1}, y^{k+1}) = Y \circ X(x^k, y^k).
\$
The local quadratic rates of Alt-NDA are shown under local Lipschitz Hessian conditions near the stationary point \cite{zhang2020newton}. Here, we introduce a modified Alt-NDA, $U_J=(X)^J\circ Y\circ X$, with additional $J\ge1$ minimization steps. We show that $U_J$ converges to $(x^\dagger,y^\dagger)$ with at least a quadratic rate. The updates of $U_J$ at iteration $k$ are as follows,
\#\label{eq:ms-newton}
& (x^{k+1, 0}, y^k) = X(x^k, y^k), \quad (x^{k+1, 0}, y^{k+1}) = Y(x^{k+1, 0}, y^k) \notag\\
& (x^{k+1, j+1}, y^{k+1}) = X(x^{k+1, j}, y^{k+1}) \text{ for } j = 0, \cdots, J-1, \quad \notag\\
& x^{k+1} = x^{k+1, J}.
\#
Let $S^\prime$ be the derivative of a mapping $S$. The following lemma provides a sufficient condition for the local quadratic convergence of any mapping. 
\begin{lemma}[Theorem 10.1.7 in \cite{ortega2000iterative}] \label{lem:ortega}
Let $S: \RR^n\to\RR^n$ and $z^\dagger$ such that $S(z^\dagger) = z^\dagger$. Suppose that $S$ is continuously differentiable on an open ball $\cB(z^\dagger,r)\subset\RR^n$ and twice differentiable at $z^\dagger$, and $S^\prime(z^\dagger)=0$. Then there is an open neighborhood $\mathfrak{N}\subset\RR^n$ of $z^\dagger$ such that for any $z^0\in\mathfrak{N}$, the iterates $\{z^k\}_{k\ge0}$ generated by $z^{k+1} = S(z^k)$ converge to $z^\dagger$ with at least a quadratic rate.
\end{lemma}

It is straightforward that Newton's method for minimization problems satisfies the above conditions and thus converges at least quadratically in a  local neighborhood. We prove the following theorem of $U_J$'s local quadratic rate based on Lemma \ref{lem:ortega}.
\begin{theorem}\label{thm:uj-local-quad}
Under Assumption \ref{ass:Hessian}, it holds  that $U_J^\prime(x^\dagger, y^\dagger) = X^\prime(x^\dagger, y^\dagger) Y^\prime(x^\dagger, y^\dagger) X^\prime(x^\dagger, y^\dagger) = 0
$ for any $J\ge 1$. 
Moreover, there is an open neighborhood $\mathfrak{N}_{U_J}$ of $(x^\dagger, y^\dagger)$ such that for any $(x^0, y^0)\in\mathfrak{N}_{U_J}$, the iterates $\{(x^{k}, y^{k})\}_{k\ge 0}$ generated by $(x^{k+1}, y^{k+1}) = U_J(x^{k}, y^{k})$ in  \eqref{eq:ms-newton} converges to $(x^\dagger, y^\dagger)$ at least quadratically.
\end{theorem}
    Theorem \ref{thm:uj-local-quad} shows the local quadratic convergence of $\{(x^{k}, y^{k})\}_{k\ge 0}$ generated by $U_J$. We note that $U_J^\prime(x^\dagger, y^\dagger) = X^\prime(x^\dagger, y^\dagger) Y^\prime(x^\dagger, y^\dagger) X^\prime(x^\dagger, y^\dagger) = 0$ holds for any $J\ge1$. Thus, when taking $J=1$ in $U_J$, two Newton's steps on $x$ in each iteration are enough to ensure a local quadratic rate.
    
    The assumption that $S$ is locally continuously differentiable
    is stronger than the local Lipschitz Hessian condition used in \cite{zhang2020newton}. Also, our updates in \eqref{eq:ms-newton} require one more minimization step per iteration compared to Alt-NDA. However, the proof of Theorem \ref{thm:uj-local-quad} is interesting as it is operator-based and significantly shorter than \cite{zhang2020newton}. Our result also generalizes previous works \cite{tapia1977diagonalized,byrd1978local} that have shown superlinear convergence of multistep Newton's update  for constrained optimization problems with Lagrangian functions. However, unfortunately, we lose the superlinear rates for the distributed hybrid methods due to the errors introduced by the distributed approximations.

\subsection{Newton's Method and its Cubic-Rate Modification}\label{sect:local-cubic}
We now recall the standard Newton's method and its local quadratic rate. Let $z=(x;y)\in\RR^{d+p}$ be the concatenation of $x$ and $y$ by column, and $\Lambda(z) = (\nabla_x L(x,y);
    \nabla_y L(x,y)):\RR^{d+p}\to\RR^{d+p}
$ be a gradient operator.
The first-order stationarity gives $\Lambda(z^\dagger)=0$. Thus, finding a first-order stationary point is equivalent to finding the root $z^\dagger$ of the system. By applying Newton's method to this root finding problem with $\nabla\Lambda(z) = \big(\begin{smallmatrix}
    \nabla_{xx}^2 L(x,y) & \nabla_{xy}^2 L(x,y)\\
    \nabla_{yx}^2 L(x,y)  & \nabla_{yy}^2 L(x,y)   
\end{smallmatrix}\big)$, we have the updates, 
\#\label{eq:z-Newton}
z^{k+1} = z^k - [\nabla\Lambda(z^k)]^{-1} \Lambda(z^k).
\#
To simplify the notation, let $\nabla_{xx}^2L$ denote $\nabla_{xx}^2L(x,y)$ and similarly for $\nabla_{xy}^2L$, $\nabla_{yx}^2L$, and $N$. By  Schur complement, with $(\nabla\Lambda)^{-1}_{11} = (\nabla_{xx}^2 L)^{-1} - (\nabla_{xx}^2 L)^{-1}(\nabla_{xy}^2 L)N^{-1}(\nabla_{yx}^2 L)(\nabla_{xx}^2 L)^{-1}$ and $(\nabla\Lambda)^{-1}_{12} = (\nabla_{xx}^2 L)^{-1}(\nabla_{xy}^2 L)N^{-1}$, the inverse $(\nabla \Lambda)^{-1}$ is given by,
\#\label{eq:schur-Hessian}
(\nabla \Lambda)^{-1} = \begin{pmatrix}
    (\nabla\Lambda)^{-1}_{11} & (\nabla\Lambda)^{-1}_{12} \\
    [(\nabla\Lambda)^{-1}_{12}]^{\intercal}  & -N^{-1}   
\end{pmatrix}.
\#
Let $\nabla_{xx}^2L^k$  denote $\nabla_{xx}^2L(x^k,y^k)$, and similarly for $\nabla_{xy}^2L^k$, $\nabla_{yx}^2L^k$, and $N^k$. By substituting \eqref{eq:schur-Hessian} to \eqref{eq:z-Newton}, we obtain the $x$ and $y$ updates in the standard Newton's method. 

We remark that it requires matrix inverses $(\nabla_{xx}^2L^k)^{-1}$ and $(N^k)^{-1}$ in each iteration, which needs the same amount of matrix inverse computation as $(\nabla_{xx}^2L^k)^{-1}$ and $[N(x^{k+1},y^k)]^{-1}$ required by Alt-NDA, as discussed in Section \ref{sect:local-AltNDA}. 

Moreover, NDA introduced in Section \ref{sect:alg-GRAND} uses $\diag((\nabla_{xx}^2L)^{-1}, -N^{-1})$ as a diagonal approximation of the Hessian inverse $(\nabla\Lambda)^{-1}$ given by \eqref{eq:schur-Hessian}. It leaves out the off-diagonal parts and a complicated multiplication on $x$'s diagonal block. NDA saves computation in terms of multiplications; thus, it might not have a quadratic rate. 

As for distributed computing, we note that  $N^{-1}$ is the most intractable part when designing hybrid methods in Section \ref{sect:dist-opt-app}. A distributed approximation of the updates in \eqref{eq:z-Newton} is computationally expensive since $N^{-1}$ is involved in each block of $(\nabla\Lambda)^{-1}$. Thus, we instead consider approximations of NDA with simple distributed implementations for solving Problems \ref{prob:dual} and \ref{prob:nf-minimax}. The multiple steps of approximations were essential to enable an easy and distributed implementation in  Algorithms \ref{alg:dish} and \ref{alg:nf}. However, their errors made it impossible to achieve a local quadratic rate even when all updates are second-order.

\vskip4pt
\noindent{\bf Modified Newton's Method with Local Cubic Rates.} The most computationally expensive step in implementing Alt-NDA and the standard Newton's method is to calculate the matrix inverses $(\nabla_{xx}^2L)^{-1}$ and $(N)^{-1}$ at each iteration. We now provide a more efficient cubically converging implementation of Newton-type updates by reusing the matrix inverse computation. We modify the Newton's method in \eqref{eq:z-Newton} by reusing the inverse $[\nabla \Lambda(z)]^{-1}$ for two consecutive steps,
\$
& z^{k+\frac{1}{2}} = z^k - [\nabla\Lambda(z^k)]^{-1} \Lambda(z^k), \\
& z^{k+1} = z^{k+\frac{1}{2}} - [\nabla\Lambda(z^k)]^{-1} \Lambda(z^{k+\frac{1}{2}}).
\$
We only need to compute $[\nabla\Lambda(z^k)]^{-1}$ once per iteration in the above updates, which involves computing the matrix inverses $(\nabla_{xx}^2L^k)^{-1}$ and $(N^k)^{-1}$. Thus, the computational cost is equivalent to that of Alt-NDA and the standard Newton's method in each iteration.  
We can rewrite the updates as,  
\#\label{eq:z-cubic}
 z^{k+1} = z^k - [\nabla\Lambda(z^k)]^{-1} [\Lambda(z^k) + \Lambda(z^k - [\nabla\Lambda(z^k)]^{-1} \Lambda(z^k))].
\#
By substituting $[\nabla\Lambda(z^k)]^{-1}$ in \eqref{eq:schur-Hessian} to \eqref{eq:z-cubic}, we can obtain an update formula for $x^{k+1}$ and $y^{k+1}$. We omit it here for simplicity. 
The following theorem shows a local cubic convergence rate of updates in \eqref{eq:z-cubic}.
\begin{theorem}[Theorem 10.2.4 in \cite{ortega2000iterative}]\label{thm:cubic}
Suppose that there is an open ball 
$\cB(z^\dagger,\widetilde{r})\subset\RR^n$ and a constant $\ell_{\Lambda}>0$ such that $\nabla\Lambda(z)$ satisfies
$
\|\nabla\Lambda(z) - \nabla\Lambda(z^\dagger)\| \le \ell_{\Lambda}\|z- z^\dagger\|$ for any $z\in \cB(z^\dagger,\widetilde{r})
$.
Suppose that $\nabla\Lambda(z^\dagger)$ is nonsingular. Then the iterates $\{z^k\}_{k\ge0}$ converge to $z^\dagger$ with a cubic rate.
\end{theorem}
Theorem \ref{thm:cubic} guarantees a much faster rate of updates in \eqref{eq:z-cubic} than Alt-NDA and the standard Newton's method, with the same computational cost in terms of the matrix inverse per iteration. This suggests we reuse the Hessian inverses and implement \eqref{eq:z-cubic} locally to achieve cubic rates in practice.

\section{Numerical Experiments}\label{sect:exp}
In this section, we conduct numerical experiments. For all problems and methods, we tune stepsizes and parameters by grid search
and select the optimal ones with the minimum number of iterations to reach a predetermined error threshold.

\subsection{Distribued Consensus Problems}\label{sect:exp-dish}
We implement DISH in Algorithm \ref{alg:dish} to solve distributed empirical risk minimization problems. We evaluate all methods on two setups, both with synthetic data. In each setup, we generate the underlying network by the Erd\H{o}s-R\'{e}nyi model with $n$ nodes and each edge
independently with probability $p$. Let $\deg_{\max} = \max_{i\in\cN}\{\deg(i)\}$ be the largest degree over the network and $Z$ be the consensus matrix with elements $z_{ii} = 1-\deg(i)/(\deg_{\max}+1)$ for $i\in\cN$, $z_{ij} = 1/(\deg_{\max}+1)$ for $\{i,j\}\in\cE$, and $z_{ij}=0$ otherwise. Let $\Theta_{i}\in \RR^{N_i\times d}$ and $v_{i}\in \RR^{N_i}$ be local feature matrix and label vector at agent $i$, respectively, and $\lambda\ge0$ be a penalty parameter. There are $N=\sum_{i\in\cN}N_i$ amount of data with local dataset size $N_i$. Let $\omega\in\RR^d$ be the decision variable. Here are the two setups. 
\vskip4pt
\emph{Setup 1: Distributed Linear Least Squares.} We study the problem that $
\min_{\omega} [(\sum_{i=1}^n  \|\Theta_{i} \omega - v_{i}\|^2)/(2N) + \lambda\|\omega\|^2/2]$ with $n=10$, $p=0.7$, $d=5$, $N_i= 50$ for $i\in\cN$, and $\lambda=1$. We generate  features $\hat \Theta_i\in\RR^{50\times 5}$, noises $u_i\in\RR^{50}$ for $i\in\cN$, and $\omega_0\in\RR^{5}$ from standard Normal distributions. We set $\Theta_i=\hat \Theta_i S$ with a scaling matrix $S=\diag\{10, 10, 0.1, 0.1, 0.1\}$ and generate $v_i\in\RR^{50}$ by $v_i = \Theta_i \omega_0 + u_i$ for $i\in\cN$.

\vskip4pt
\emph{Setup 2: Distributed Logistic Regression.} 
For $v_{i}\in \{0,1\}^{N_i}$ and $h_i= 1/(1+\exp(-\Theta_i\omega))$, we study the problem $\min_{\omega} [(\sum_{i=1}^n [- v_i^\intercal \log h_i - (1-v_i)^\intercal\log(1-h_i)])/N + \lambda\|\omega\|^2/2]$. We set $n=20$, $p=0.5$, $d=3$, $N_i= 50$ for $i\in\cN$, and $\lambda=1$. We generate $\hat \Theta_i\in\RR^{50\times 3}$, noises $u_i\in\RR^{50}$ for $i\in\cN$, and $\omega_0\in\RR^{3}$ from Normal distributions. We scale $\hat \Theta_i$ with $S=\diag\{10, 0.1, 0.1\}$ and set feature matrices to be $\Theta_i=\hat \Theta_i S$. Moreover, we generate $v_i\in\RR^{50}$ by the formula $v_i = \argmax(\text{softmax}(\Theta_i \omega_0 + u_i))$.

We compare EXTRA \cite{shi2015extra}, ESOM-$0$ \cite{mokhtari2016decentralized}, and different variants of DISH in Algorithm \ref{alg:dish} for the two setups. Let DISH-$K$ represent DISH with $K$ agents consistently performing Newton-type updates while others adopt gradient-type updates. DISH-G$\&$N denote DISH with all agents switching between gradient-type and Newton-type updates occasionally. In particular, DISH-G$\&$N-U and DISH-G$\&$N-LN denote agents changing their update types every $t_i$ iterations, where $t_i\sim U[5,50]$ and $t_i\sim\text{lognormal}(2,4)+30$, respectively. The initial updates for DISH-G$\&$N-U and DISH-G$\&$N-LN are uniformly sampled from $\{$`gradient-type', `Newton-type'$\}$. The error is measured by $\|x^k-x^\star \|/\|x^0-x^\star\|$, where  $x^\star$ is the optimal solution obtained by a centralized solver. In DISH, we fix $a_i=1$ for Newton-type updates to mimic the primal Newton's step.

\begin{figure}[!t]
  \centering
    \subfloat[][\small Least Squares in Setup 1]{%
         \includegraphics[width=0.4\linewidth]{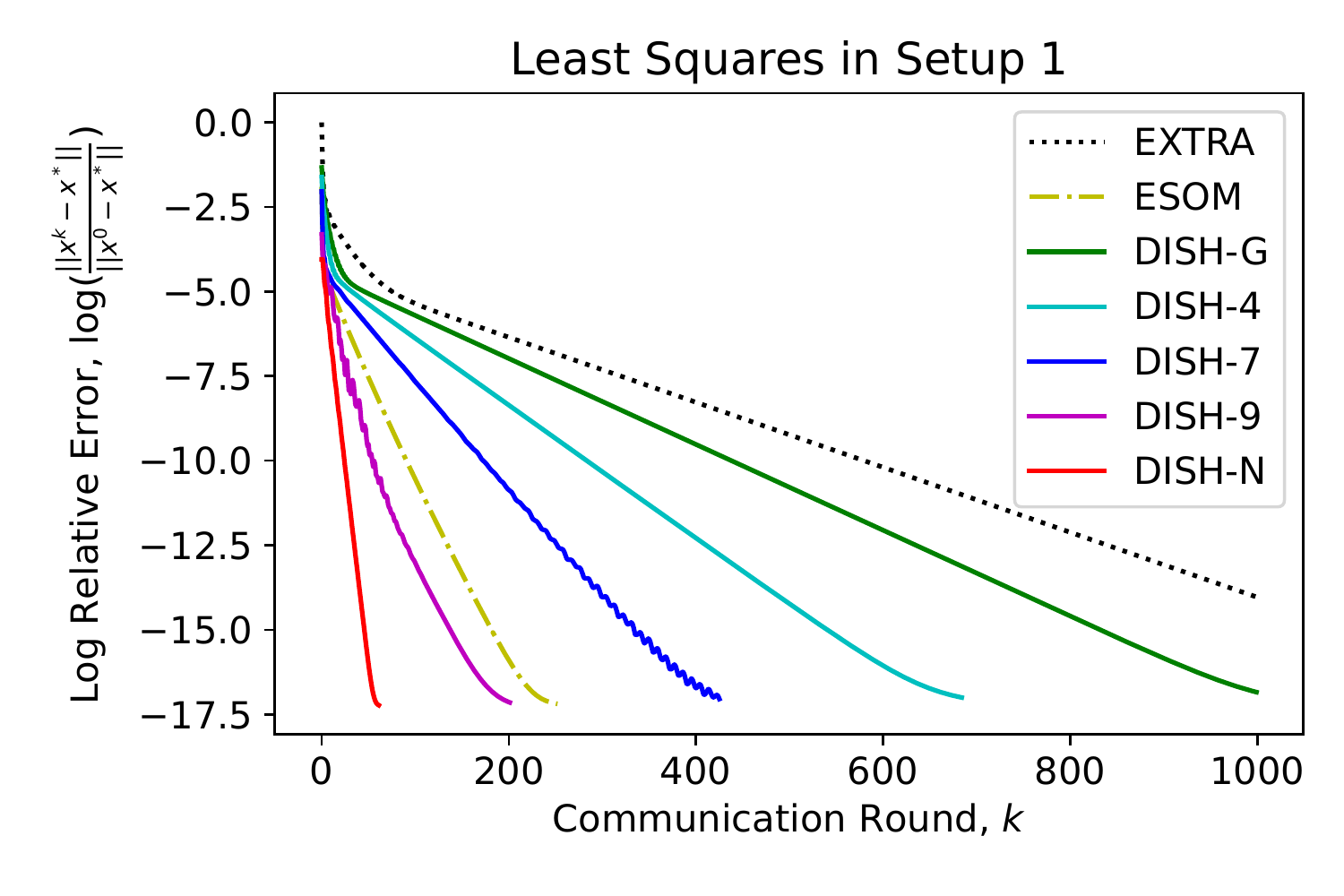}}
    \subfloat[][\small Logistic Regression in Setup 2]{%
         \includegraphics[width=0.4\linewidth]{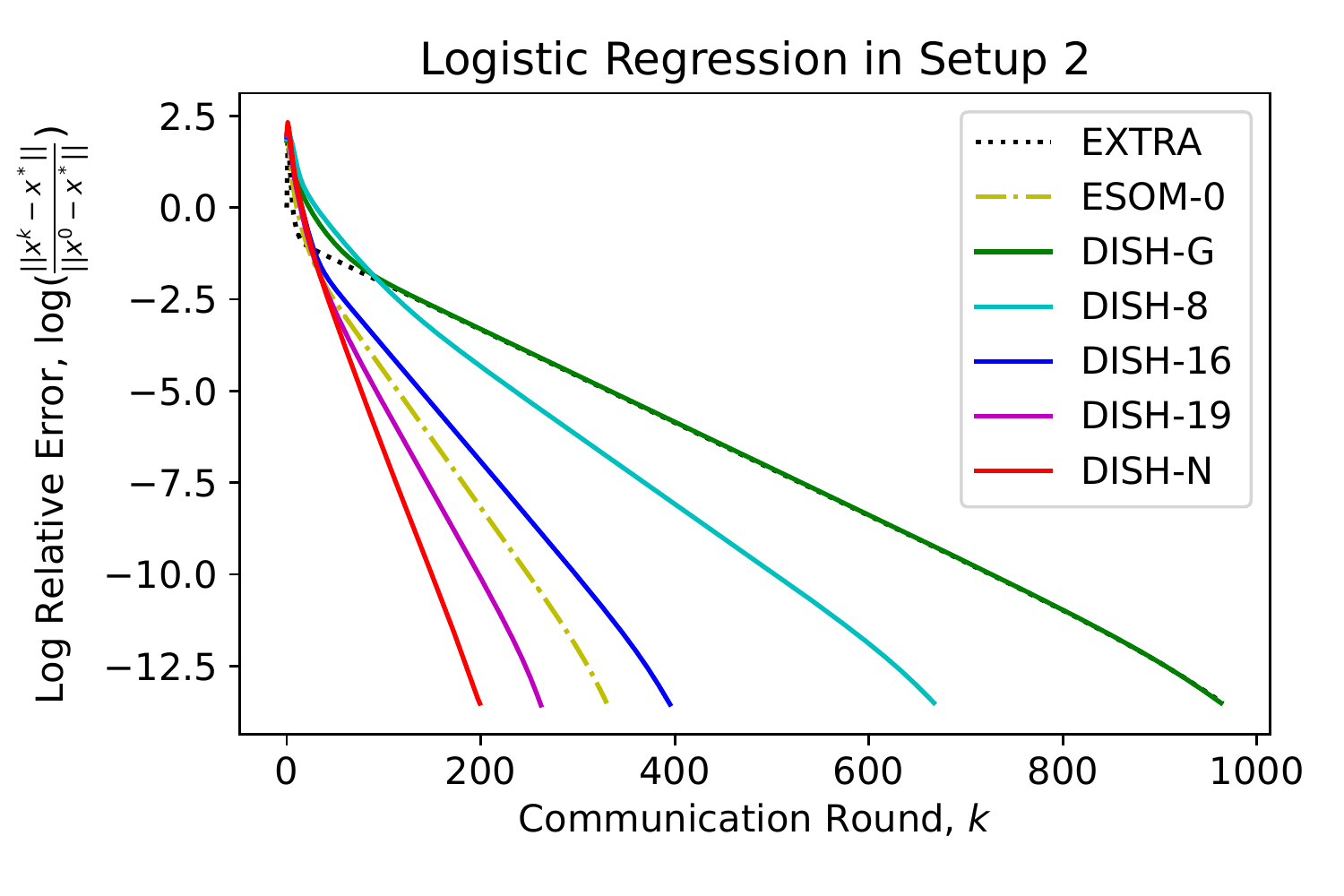}} \\
    \subfloat[][\small Least Squares in Setup 1]{%
         \includegraphics[width=0.4\linewidth]{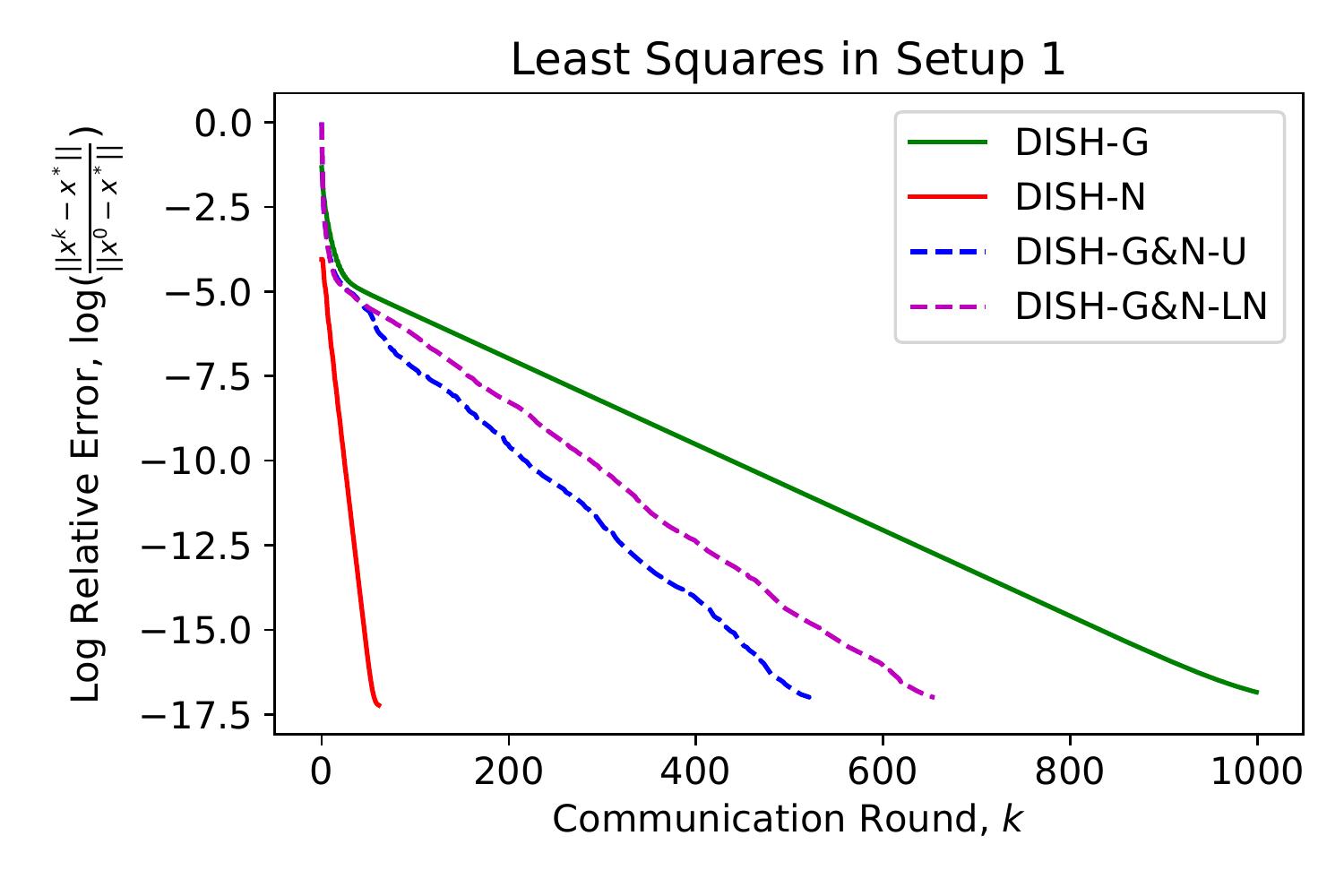}}
    \subfloat[][\small Logistic Regression in Setup 2]{%
         \includegraphics[width=0.4\linewidth]{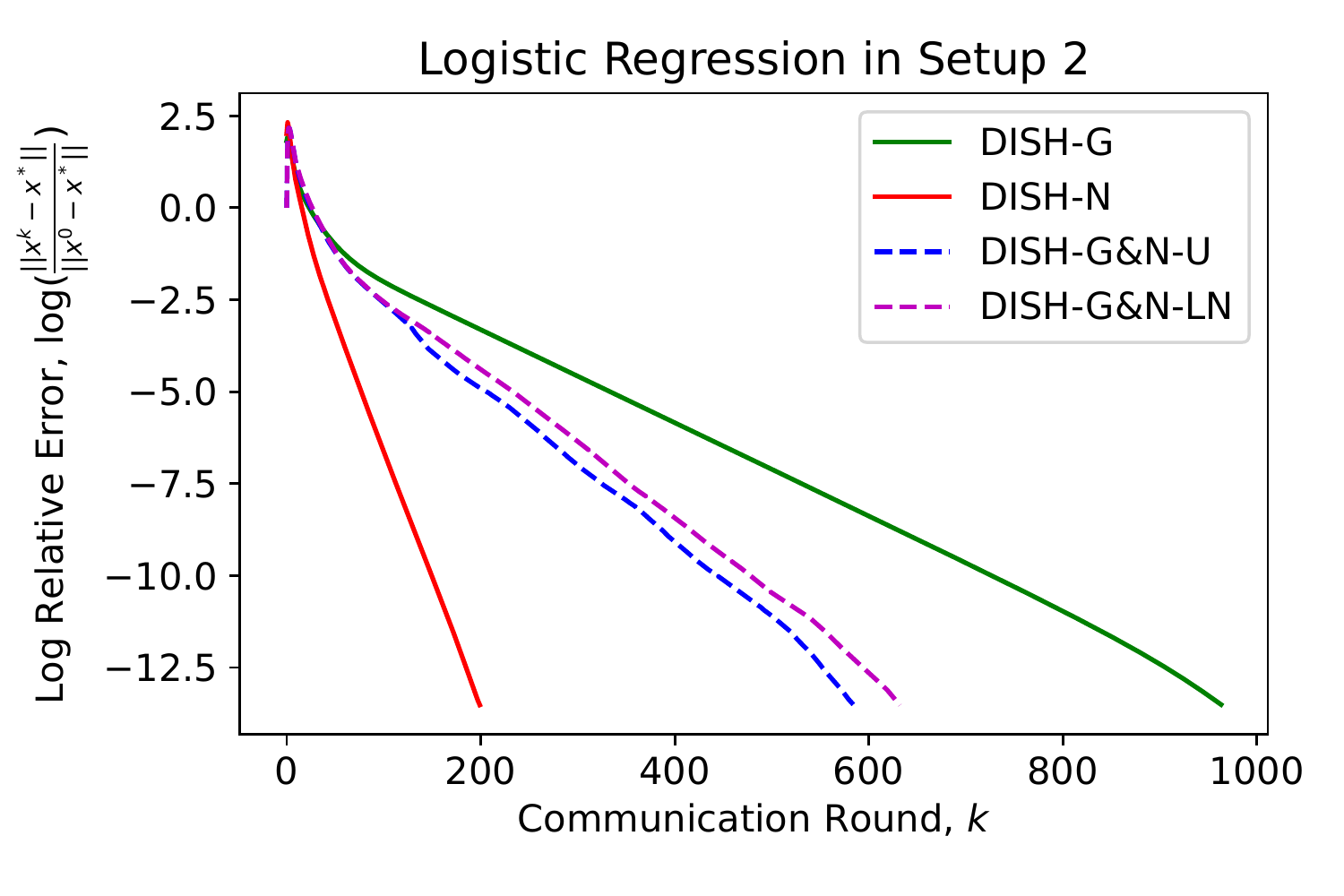}}
         \caption{Results of EXTRA, ESOM-$0$, DISH, and DISH-G\&N.}
       \label{fig:dish_1}
  \end{figure}

All methods in this study require one communication round with the same communication costs per iteration regardless of the update type. Figure \ref{fig:dish_1} depicts the number of communication rounds (iterations) on the $x$-axis and the logarithm of the relative error on the $y$-axis. The results demonstrate that DISH achieves linear performance regardless of the agents' choice of gradient-type and Newton-type updates, validating the theoretical guarantees presented in Theorem \ref{thm:PL}. 
Notably, the performance of the first-order methods, such as EXTRA and DISH-G, is similar. However, when some agents adopt Newton-type updates, DISH consistently outperforms the baseline method DISH-G, resulting in faster training. Specifically, DISH-N outperforms ESOM-$0$ in various scenarios, indicating the benefits of dual Hessian approximation in DISH-N. Additionally, increasing the number of agents performing Newton-type updates (denoted by $K$) tends to accelerate the convergence of DISH by leveraging more Hessian information. This observation suggests that in practical scenarios, agents with higher computational capabilities or cheaper computation costs can locally implement Newton-type updates to enhance the overall convergence speed of the system. 

\subsection{Network Flow Problems}

We evaluate the numerical performance of Algorithm \ref{alg:nf} on network flow problems. We generate an Erd\H{o}s-R\'{e}nyi network with $n=10$ nodes, where each edge exists with probability $p=0.4$. The resulting connected graph has $18$ edges. We study a network flow problem, $\min_x [\sum_{\{i,j\}\in\cE} (h_{ij}x_{ij} - v_{ij})^2]/2$ such that $Ex=s$, where $x\in\RR^{18}$ is the decision variable, and $E\in\RR^{10\times 18}$ is the incidence matrix of the graph. We generate vectors $h\sim\text{lognormal}(3,1), v\sim\text{lognormal}(1,1)\in\RR^{18}$, and $\widehat{s}\sim U(0,1)\in\RR^{10}$. To ensure feasibility, we set $s_i=\widehat{s}_i - (\sum_{i\in\cN}\widehat{s}_i)/n$ so that $\sum_{i\in\cN} s_i = 0$. 
\begin{figure}[!t]
    \centering
    \includegraphics[width=0.4\linewidth]{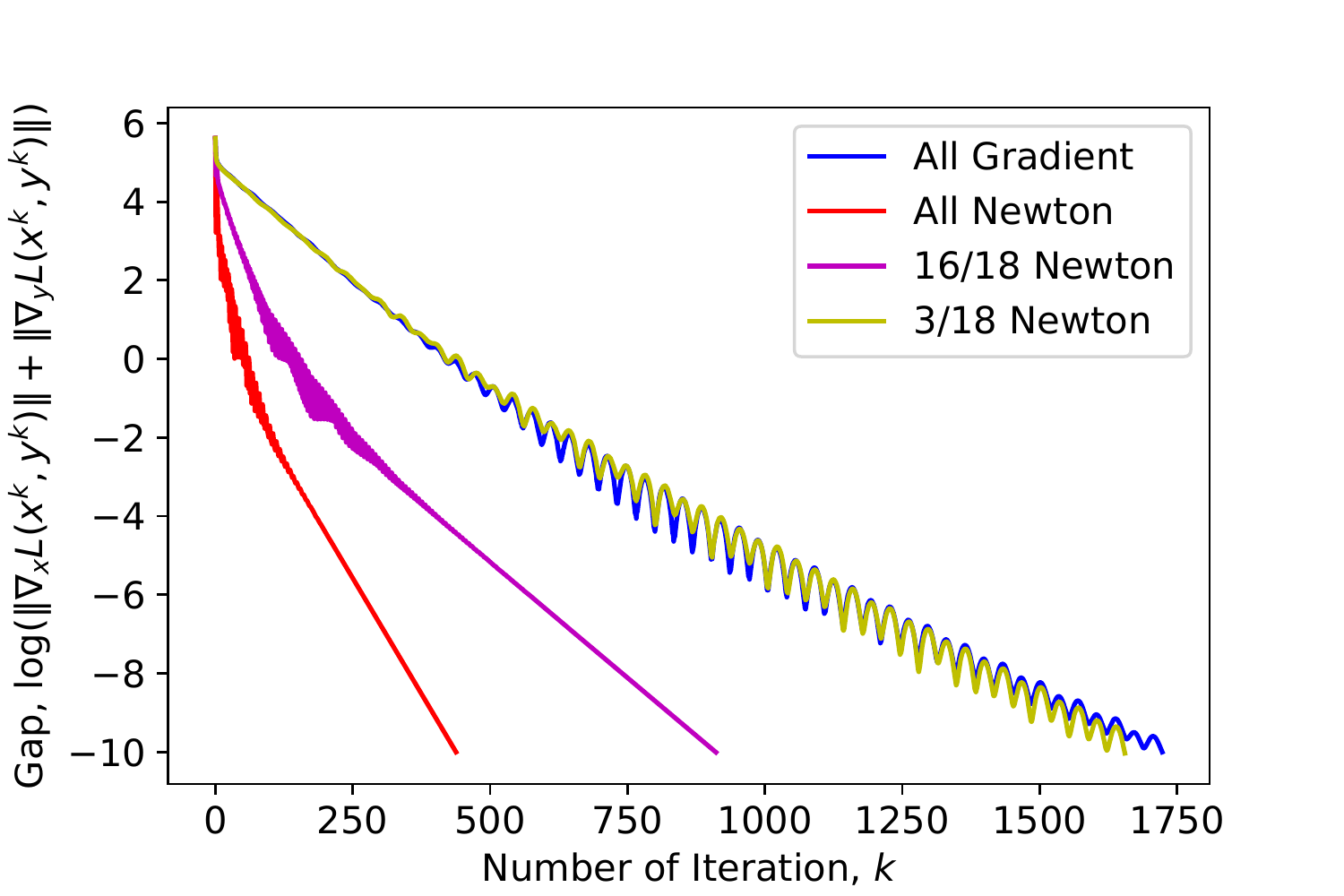}
    \caption{Network Flow Problems}
    \label{fig:NF}
\end{figure}
We evaluate Algorithm \ref{alg:nf} with different numbers of Newton-type edges to solve the problem. Specifically, we compare cases where all edges take gradient-type (or Newton-type) updates, and cases where $3$ or $16$ (out of $18$) edges take Newton-type updates. Figure \ref{fig:NF} illustrates the optimal performance of these cases, where the $y$-axis is the logarithm of the Lagrangian norms.
The results show the linear convergence of Algorithm \ref{alg:nf} regardless of the agents' update choices. Notably, the all-Newton case exhibits significantly faster convergence, highlighting the effectiveness of the diagonal approximation of the dual Hessian presented in Section \ref{sect:nf_nda}. Additionally, increasing the number of Newton-type agents generally improves the overall performance.
\subsection{Centralized Minimax Problems}\label{sect:exp-centralized}
We compare the performance of methods on centralized minimax problems. Specifically, we consider GDA and NDA within the GRAND framework and Alt-GDA and Alt-NDA within the Alt-GRAND framework. We also include the optimistic GDA (OGDA) \cite{mokhtari2020unified} for comparison, which incorporates negative momentum into the gradient updates. Given $A\in\RR^{n\times d}$, $b\in\RR^n$, $a>0$, and $\lambda>0$, we study the strongly-convex-concave problem that
$\max_{y\in\RR^d}\min_{x\in\RR^n}[(\|x\|^2/2 
+ b^\intercal x + x^\intercal Ay)/n - \lambda R_a(y)]$,
 where $R_a(y) = \sum_{i=1}^d[\log(1+\exp(ay_i)) + \log(1+\exp(-ay_i))]/a$ is Lipschitz smooth and convex.
The problem is a minimax reformulation \cite{du2019linear} of the linear regression problem with smoothed-$L_1$ regularization,
$\min_{\omega\in\RR^d} [\|A\omega-b\|^2/(2n) + \lambda R_a(\omega)]$. We use a California housing dataset for regression, with $d=9$, $n=14448$ (training samples), $a=10$, and $\lambda=1/n$.
We initialize all methods with $(x^0, y^0) = (0, 0)$ and measure optimality using Lagrangian gradient norms.

Figure \ref{fig:centralized} shows the optimal performance of the methods. OGDA outperforms GDA. NDA achieves much faster convergence compared to GDA and OGDA due to its utilization of second-order information. The alternating methods, Alt-GDA and Alt-NDA, perform better than their standard counterparts. Note that in GRAND methods, both $x$ and $y$ are updated simultaneously, while in Alt-GRAND methods, $y$ has to wait for the $x$ update. A rough estimate suggests that each iteration of Alt-GRAND takes twice as much time as GRAND. However, our results show that GRAND requires more than two times iterations to converge compared to Alt-GRAND. Thus, Alt-GRAND tends to be more time-efficient overall. 

In the second experiment, we compare the local performance of Newton-type methods discussed in Section \ref{sect:local-Newton}. We study the strongly-convex-concave problem $\max_{y\in\RR^d}\min_{x\in\RR^n}[\|x\|_4^2/2 
+ b^\intercal x + x^\intercal Ay - \lambda R_a(y)]$, which is a minimax reformulation of the problem $\min_{\omega}[\|A\omega-b\|^2_{4/3}/2 + \lambda R_a(\omega)]$. We set $n=5$, $d=20$, $a=10$, and $\lambda=1$, and generate each row of $A$ as $A_i\sim \cN(0, I_d)$ and $b\sim\text{lognormal}(0,1)$. We run Newton-type methods, including NDA, Alt-NDA, Twostep-Alt-NDA in \eqref{eq:ms-newton} with $J=1$, Newton's method, and the cubic method in \eqref{eq:z-cubic}.
Starting from a point close to the optimal solution obtained by running Alt-GDA for a few rounds, we continue until the gradient norms reach machine precision. We tune the stepsizes in NDA using grid search while fixing the stepsizes in the other methods to be $1$. Figure \ref{fig:centralized} shows the results of this experiment. NDA exhibits linear rates, while the other methods achieve much faster convergence.  The cubic method outperforms the others though the cubic rate is 
not obvious. Also, the cubic method is more sensitive to the initial point than the others. 

\begin{figure}[!t]
  \centering
    \subfloat[][\small (Alt-)GRAND and OGDA]{%
      \includegraphics[width=0.4\linewidth]{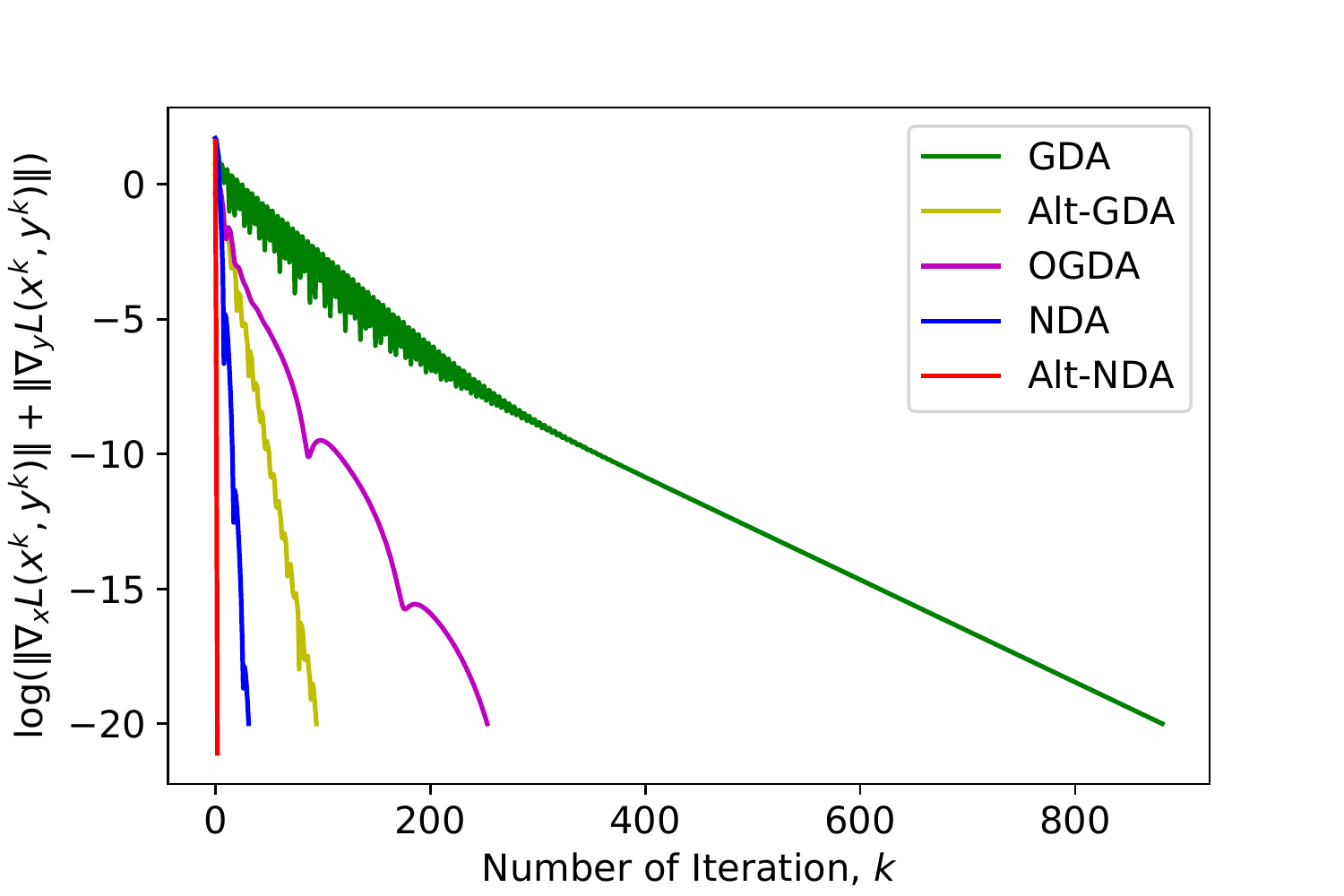}}
    \subfloat[][\small Newton-Type Methods]{%
       \includegraphics[width=0.4\linewidth]{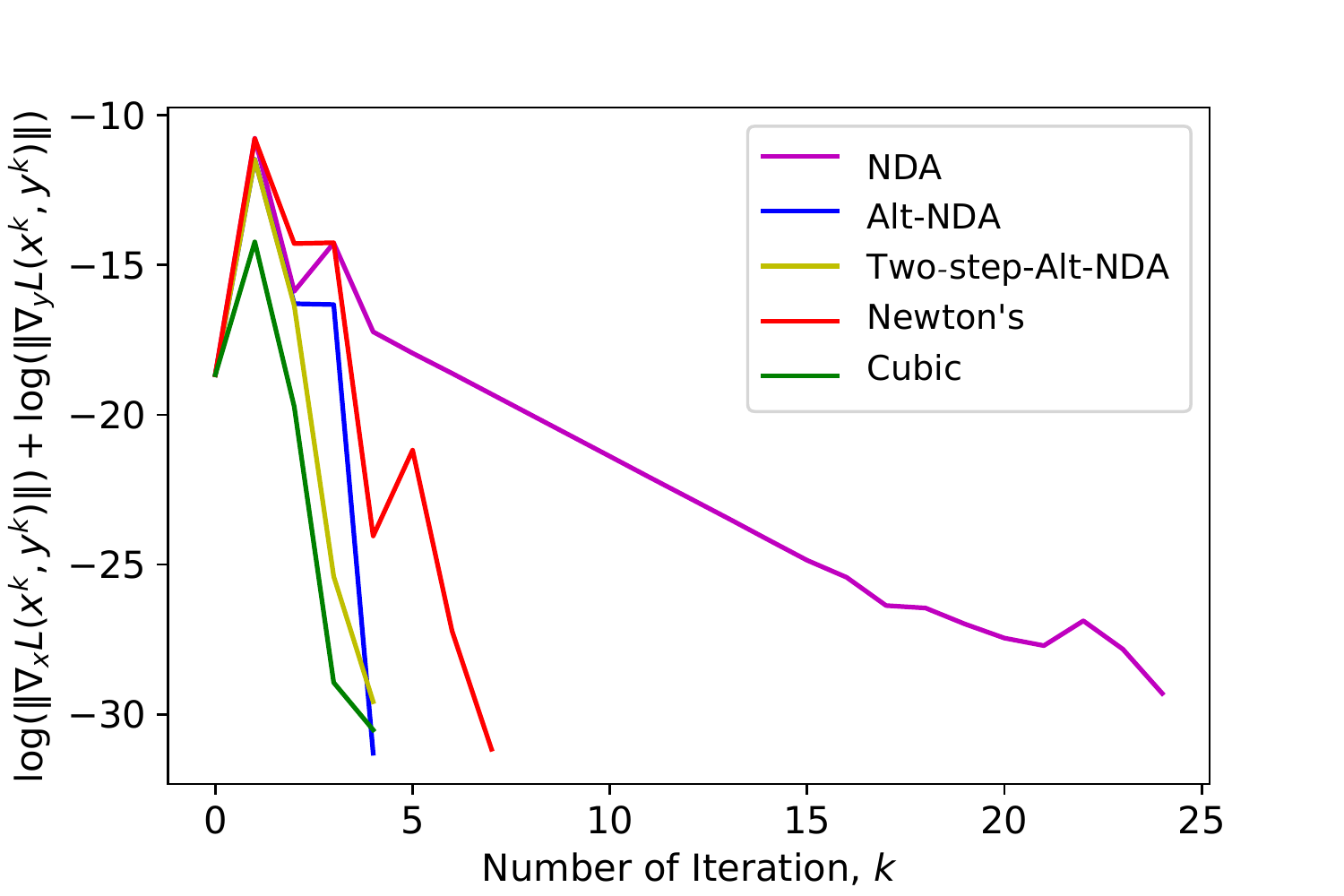}}
         \caption{Centralized Minimax Problems.}
       \label{fig:centralized}
  \end{figure}

\section{Conclusions}
This work proposes DISH, a distributed hybrid method that leverages agents' computational heterogeneity. DISH allows agents to choose between gradient-type and Newton-type updates, improving overall efficiency. GRAND is introduced to analyze the performance of methods with general update directions. Theoretical analysis shows global rates for GRAND, ensuring linear convergence for DISH. Future work directions include applications to nonconvex and stochastic settings. 

\bibliographystyle{IEEEtran}
\bibliography{ref}

\appendix

\section{Proofs in Section \ref{sect:dist-opt-app}}
We now present the proofs of lemmas and theorems in the appendix.
\begin{proof}[Proof of Lemma \ref{lem:psi-grad-hess}]
    We prove Lemma \ref{lem:psi-grad-hess} under Assumption \ref{ass:Hessian}.
    The definition of $x^*(y)$ in \eqref{eq:x_star_y} and the first-order optimality condition of the minimization problem give that, for any $y\in\RR^p$, 
    \#\label{eq:opt_x_star_y}
    0 = \nabla_xL(x,y)\given_{x = x^*(y)} = \nabla_xL(x^*(y),y).
    \#
    Since $\nabla_x L(x,y)$ is continuously differentiable relative to $(x,y)$ and $\nabla_{xx}^2 L(x,y)\succ 0$ for any $(x,y)$ under Assumption \ref{ass:Hessian}, by the implicit function theorem, the function $x^*(y)$ is continuously differentiable relative to $y$. Its derivative $\nabla x^*(y)$ is determined by differentiating \eqref{eq:opt_x_star_y} on both sides, 
    \$
    0 = \nabla_{xx}^2L(x^*(y),y)\nabla x^*(y) + \nabla_{xy}^2L(x^*(y),y).
    \$
    Since $\nabla_{xx}^2L(x^*(y),y)$ is always invertible under Assumption \ref{ass:Hessian}, by rearranging the terms, we have 
    \#\label{eq:implic_func}
    \nabla x^*(y) = - [\nabla_{xx}^2L(x^*(y),y)]^{-1}\nabla_{xy}^2L(x^*(y),y).
    \#
    Moreover, the definition of $\psi(y)$ in \eqref{eq:x_star_y} and the chain rule give that
    \$
    \nabla \psi(y) &= \diff{L(x^*(y),y)}{y}\notag\\
    & = \big(\nabla x^*(y)\big)^\intercal\nabla_x L(x^*(y),y) + \nabla_y L(x^*(y),y) \notag\\
    & = \nabla_y L(x^*(y),y),
    \$
    where the last equality is due to \eqref{eq:opt_x_star_y}. By further differentiating relative to $y$ on both sides, we have
    \$
    \nabla^2\psi(y) = \nabla_{yx}^2 L(x^*(y),y) \nabla x^*(y) + \nabla_{yy}^2 L(x^*(y),y).
    \$
    Finally, by substituting \eqref{eq:implic_func} in the preceding equation, we have
    \$
    \nabla^2\psi(y) = - \nabla_{yx}^2 L(x^*(y),y) [\nabla_{xx}^2L(x^*(y),y)]^{-1}\nabla_{xy}^2L(x^*(y),y) + \nabla_{yy}^2 L(x^*(y),y).  
    \$
    This concludes the form of $\nabla \psi(y)$ and $\nabla^2 \psi(y)$.
\end{proof} 

\begin{proof}[Proof of Proposition \ref{prop:feature}]
    By introducing $z_i = \Theta_i \xi_i \in \RR^N$ for $i\in\cN$, we rewrite Problem \ref{prob:feature} as,
    \#\label{prob:feature_rewrite}
    \min_{\xi\in\RR^d} \phi\big(\sum_{i=1}^n z_i\big) + \sum_{i=1}^n r_i(\xi_i), \quad \text{s.t. } z_i = \Theta_i \xi_i, \text{ for } i\in\cN.
    \#
    Then we define $x_i\in\RR^N$ as the dual variable associated with the constraint $z_i = \Theta_i \xi_i $ at agent $i$ for $i\in\cN$. 
    We define the Lagrangian function $L^{\sf fp}(\xi,x)$ of Problem \ref{prob:feature_rewrite} as follows,
    \$
    L^{\sf fp}(\xi,x) = \phi\big(\sum_{i=1}^n z_i\big) + \sum_{i=1}^n r_i(\xi_i) + \sum_{i=1}^n x_i^\intercal (\Theta_i \xi_i - z_i).
    \$
    Since both $\phi$ and $r$ are convex, strong duality holds for Problem \ref{prob:feature_rewrite} by the Slater’s condition. Thus, Problem \ref{prob:feature_rewrite} is equivalent to the following dual problem,
    \#\label{prob:dual-feature}
    \max_{x_1,\cdots,x_n} \min_{\xi} L^{\sf fp}(\xi,x).
    \#
    By straightforward calculation, we have
    \#\label{eq:dual-feature}
    \min_{\xi} L(\xi,x) & = \min_{\xi}\Big[ \phi\big(\sum_{i=1}^n z_i\big) + \sum_{i=1}^n r_i(\xi_i) + \sum_{i=1}^n x_i^\intercal (\Theta_i \xi_i - z_i)\Big] \notag \\
    & = \min_{\xi}\Big[ \phi\big(\sum_{i=1}^n z_i\big) - \sum_{i=1}^n x_i^\intercal z_i \Big] + \sum_{i=1}^n\min_{\xi}\big[ r_i(\xi_i)  + x_i^\intercal\Theta_i \xi_i\big]
    \#
    We note that by the definition of the convex conjugate, we have
    \$
    &\min_{\xi}\Big[ \phi\big(\sum_{i=1}^n z_i\big) + \sum_{i=1}^n x_i^\intercal z_i \Big] = \begin{cases} - \phi^*(x_1),  \quad \text{if } x_1=\cdots=x_n\\
        -\infty,  \quad \text{otherwise}
    \end{cases}, \\
    &\min_{\xi}\big[ r_i(\xi_i)  + x_i^\intercal\Theta_i \xi_i\big] = -r_i^*(-\Theta_i^\intercal x_i).
    \$
    Thus, by substituting the preceding relation in \eqref{eq:dual-feature}, Problem \ref{prob:dual-feature}, as well as Problem \ref{prob:feature_rewrite}, is equivalent to the following problem,
    \$
    \max_{x_1,\cdots,x_n} - \phi^*(x_1) - \sum_{i=1}^n r_i^*(-\Theta_i^\intercal x_i), \quad \text{s.t. } x_1 = \cdots = x_n.
    \$
    By flipping the sign, the above problem is equivalent to the problem stated in the proposition.
    \end{proof}

    \section{Proofs in Section \ref{sect:global}}
    \begin{proof}[Proof of Lemma \ref{lem:psi_property}]
        Due to the definition of $x^*(y)$ in \eqref{eq:x_star_y} and the first-order optimality condition of the minimization problem, for any $y\in\RR^p$, we have 
        \$
        0 = \nabla_xL(x,y)\given_{x = x^*(y)} = \nabla_xL(x^*(y),y).
        \$
        To show the Lipschitz continuity of $\nabla \psi(y)$, we first prove that $x^*(y)$ as a function of $y$ is $(\ell_{\sf xy}/m_{\sf x})$-Lipschitz continuous.
        For any $y,z\in\RR^p$, by the preceding equation, we have
        \$
        0 & = \nabla_xL(x^*(y),y) - \nabla_xL(x^*(z),z) \\
        & = \nabla_xL(x^*(y),y) - \nabla_xL(x^*(z),y) + \nabla_xL(x^*(z),y) - \nabla_xL(x^*(z),z).
        \$
        By rearranging the terms in the above equation, we have 
        \$
        \|\nabla_xL(x^*(y),y) - \nabla_xL(x^*(z),y)\| & = \|\nabla_xL(x^*(z),y) - \nabla_xL(x^*(z),z)\|.
        \$
        Therefore, for any $y,z\in\RR^p$, by the $m_{\sf x}$-strong convexity of $L(x, y)$ with respect to $x$, we have
        \#\label{eq:x_star_lip}
        \|x^*(y) - x^*(z)\| &\le \frac{1}{m_{\sf x}}\|\nabla_xL(x^*(y),y) - \nabla_xL(x^*(z),y)\| \notag\\
        & = \frac{1}{m_{\sf x}} \|\nabla_xL(x^*(z),y) - \nabla_xL(x^*(z),z)\| \notag\\
        &\le \frac{\ell_{\sf xy}}{m_{\sf x}} \|y-z\|,
        \#
        where the equality is due to the preceding equation and the inequality holds since $\nabla_x L(x^*(y),y)$ is $\ell_{\sf xy}$-Lipschitz continuous on $y$ under Assumption \ref{ass:Hessian}.
        Thus, Lemma \ref{lem:psi-grad-hess} yields for any $y,z\in\RR^p$,
        \$
        \|\nabla \psi(y) - \nabla\psi(z)\| & = \|\nabla_y L(x^*(y),y) - \nabla_y L(x^*(z),z)\| \\
        & \le \|\nabla_y L(x^*(y),y) - \nabla_y L(x^*(y),z)\| + \|\nabla_y L(x^*(y),z) - \nabla_y L(x^*(z),z)\| \\
        &\le \ell_{\sf yy}\|y-z\| + \ell_{\sf yx}\|x^*(y)-x^*(z)\| \\
        &\le \big(\ell_{\sf yy} + \frac{\ell_{\sf yx}\ell_{\sf xy}}{m_{\sf x}}\big)\|y-z\|,
        \$
        where the second inequality holds due to the $\ell_{\sf yy}$ and $\ell_{\sf yx}$-Lipschitz continuity of $\nabla_y L(x,y)$ with respect to $y$ and $x$, respectively, and the last inequality follows from the $(\ell_{\sf xy}/m_{\sf x})$-Lipschitz continuity of $x^*(y)$ in \eqref{eq:x_star_lip}.
        \end{proof}

        \begin{proof}[Proof of Lemma \ref{lem:descent_direct}]
            We first show the result for $s^k$. Under Assumption \ref{ass:descent_direct}, we have
            \$
            (s^k)^\intercal\nabla_x L(x^k,y^k) &\ge \frac{1}{\Gamma_s}\|s^k\|^2 \\
            &\ge \frac{1}{\Gamma_s} (\sqrt{\gamma_s\Gamma_s} \|\nabla_x L(x^k,y^k)\|)^2 \\
            & = \gamma_s \|\nabla_x L(x^k,y^k)\|^2.
            \$
            Moreover, it also holds that
            \$
            \frac{1}{\Gamma_s}\|s^k\|^2 \le (s^k)^\intercal\nabla_x L(x^k,y^k) \le \|s^k\| \|\nabla_x L(x^k,y^k)\|.
            \$
            By dividing $\|s^k\|$ on both sides of the above relation, we have
            \$
            \|s^k\|\le \Gamma_s\|\nabla_x L(x^k,y^k)\|.
            \$
            By combining the two relations, we have 
            \$
            \gamma_s \|\nabla_x L(x^k,y^k)\|^2 \le (s^k)^\intercal\nabla_x L(x^k,y^k) \le \|s^k\|\|\nabla_x L(x^k,y^k)\|\le \Gamma_s\|\nabla_x L(x^k,y^k)\|^2.
            \$
            Thus, we have $\gamma_s\le\Gamma_s$. This concludes the results for $s^k$. Since the assumptions and the results for $t^k$ have the same structure as those for $s^k$, the results also hold for $t^k$.
            \end{proof}
            
            \begin{proof}[Proof of Lemma \ref{lem:wx_sc}]
            Due to the definition of $x^*(y)$ in \eqref{eq:x_star_y} and the first-order optimality condition of the minimization problem, for any $y\in\RR^p$, we have 
            \$
            0 = \nabla_xL(x,y)\given_{x = x^*(y)} = \nabla_xL(x^*(y),y).
            \$
            Thus, due to the $m_{\sf x}$-strong convexity of $L(x,y)$ with respect to $x$ under Assumption \ref{ass:Hessian}, we have
            \$
            \big\|\nabla_xL(x, y)\big\| = \big\|\nabla_xL(x, y) - \nabla_xL(x^*(y), y)\big\| \ge m_{\sf x}\|x - x^*(y)\|.
            \$
            Therefore, by the explicit form of $\nabla \psi(y)$ in Lemma \ref{lem:psi_property} and the $\ell_{\sf yx}$-Lipschitz continuity of $\nabla_yL(x,y^k)$ with respect to $x$ under Assumption \ref{ass:Hessian}, we have
            \$
            \big\|\nabla_y L(x^k,y^k) - \nabla \psi(y^k)\big\|& =\|\nabla_y L(x^k,y^k)-\nabla_y L(x^*(y^k),y^k)\|  \\
            & \le \ell_{\sf yx}\|x^k-x^*(y^k)\| \\
            &\le\frac{\ell_{\sf yx}}{m_{\sf x}}\big\|\nabla_xL(x^k, y^k)\big\|,
            \$
            where the last inequality is due to the preceding relation.
            \end{proof}
            
            \begin{proof}[Proof of Lemma \ref{lem:de_lambda}]
            By the inequality that $\|a+b\|^2 \le (1+\upsilon)\|a\|^2 + (1+1/\upsilon)\|b\|^2$ for any $a,b\in\RR^p$ and $\upsilon>0$, we have
            \$
            \big\|\nabla_y L(x^k,y^k)\big\|^2
            & \le (1+\upsilon) \big\|\nabla \psi(y^k)\big\|^2 + \big(1+\frac{1}{\upsilon}\big) \big\|\nabla_y L(x^k,y^k)-\nabla \psi(y^k)\big\|^2 \\
            & \le (1+\upsilon) \big\|\nabla \psi(y^k)\big\|^2 + \big(1+\frac{1}{\upsilon}\big)\frac{\ell_{\sf yx}^2}{m_{\sf x}^2}\big\|\nabla_xL(x^k,y^k)\big\|^2,
            \$
            where the last inequality follows from Lemma \ref{lem:wx_sc}. Moreover, the dual update in Algorithm \ref{alg:grand} gives
            \$
            \|y^{k+1} - y^k\|^2 & = \beta^2\|t^k\|^2  \le \beta^2\Gamma_t^2\big\|\nabla_y L(x^k,y^k)\big\|^2 \\
            & \le 2\beta^2\Gamma_t^2 \big\|\nabla \psi(y^k)\big\|^2 + \frac{2\beta^2\Gamma_t^2\ell_{\sf yx}^2}{m_{\sf x}^2}\big\|\nabla_xL(x^k,y^k)\big\|^2,
            \$
            where the inequalities follow from Lemma \ref{lem:descent_direct} and  the preceding relation with $\upsilon=1$.
            \end{proof}

            \begin{proof}[Proof of Proposition \ref{prop:gL}]
                Due to the $\ell_\psi$-Lipschitz continuity of $\nabla\psi$ in Lemma \ref{lem:psi_property} \cite{nesterov2018lectures}, we have
                \#\label{eq:step1}
                \psi(y^{k+1}) &\ge \psi(y^k) + \inner{\nabla \psi(y^{k})}{y^{k+1} - y^k} - \frac{\ell_\psi}{2}\|y^{k+1} - y^k\|^2 \notag \\
                & = \psi(y^k) + \beta\inner{\nabla \psi(y^{k})}{t^k} - \frac{\ell_\psi}{2}\|y^{k+1} - y^k\|^2,
                \#
                where the equality is due to the $y$-update in Algorithm \ref{alg:grand}. We define a constant $\rho=\Gamma_t^2/\gamma_t$.
                We now lower bound the second term $\langle {\nabla \psi(y^{k})},{t^k}\rangle$ in \eqref{eq:step1}. We first consider the case when $\gamma_t<\Gamma_t$. In this case, we have,
                \#\label{eq:t_rho_pyL}
                \|t^k- \rho \nabla_{y}L(x^k,y^k)\|^2 & = \|t^k\|^2 - 2 \rho (t^k)^\intercal \nabla_{y}L(x^k,y^k) + \rho^2\|\nabla_{y}L(x^k,y^k)\|^2 \notag \\
                & \le (\Gamma_t^2 - 2\rho \gamma_t + \rho^2)\|\nabla_{y}L(x^k,y^k)\|^2 \notag \\
                & \le (\Gamma_t^2 - 2\rho \gamma_t + \rho^2)\big[(1+\nu)\|\nabla \psi(y^k)\|^2 + (1+\frac{1}{\nu})\frac{\ell_{\sf yx}^2}{m_{\sf x}^2}\|\nabla_xL(x^k,y^k)\|^2\big] \notag \\
                & = \frac{\rho^2}{(1+\nu)^2}\|\nabla \psi(y^k)\|^2 + \frac{\rho^2\ell_{\sf yx}^2}{\nu(1+\nu)^2m_{\sf x}^2}\|\nabla_xL(x^k,y^k)\|^2,
                \#
                where the inequalities are due to Lemma \ref{lem:descent_direct} and Lemma \ref{lem:de_lambda} with $\upsilon = \nu >0$, respectively, and the last equality is due to the definitions of $\nu$ and $\rho$. 
                Now we define positive constants $\eta_1 = \rho/(1+\nu)$ and $\eta_2 = \rho\nu/(1+\nu)$. We can lower bound $\langle {\nabla \psi(y^{k})},{t^k}\rangle$ by adding and subtracting some terms,
                \#\label{eq:lambdainner}
                & \langle\nabla \psi(y^{k}), t^k\rangle  = \langle\nabla \psi(y^{k}),t^k - \rho\nabla_{y}L(x^k,y^k)\rangle + \langle\nabla \psi(y^{k}), \rho\nabla_{y}L(x^k,y^k) - \rho\nabla \psi(y^{k})\rangle + \rho \|\nabla \psi (y^k)\|^2 \notag\\
                &\qquad \ge (\rho-\frac{\eta_1}{2} - \frac{\eta_2}{2})\|\nabla \psi(y^{k})\|^2 - \frac{1}{2\eta_1}\|t^k- \rho \nabla_{y}L(x^k,y^k)\|^2 - \frac{\rho^2}{2\eta_2}\|\nabla_{y}L(x^k,y^k) - \nabla \psi(y^{k})\|^2 \notag \\
                &\qquad \ge \big[\rho-\frac{\eta_1}{2} - \frac{\eta_2}{2} - \frac{\rho^2}{2\eta_1(1+\nu)^2}\big]\|\nabla \psi(y^{k})\|^2 - \big[\frac{\rho^2}{2\eta_1\nu(1+\nu)^2} + \frac{\rho^2}{2\eta_2}\big]\frac{\ell_{\sf yx}^2}{m_{\sf x}^2}\|\nabla_xL(x^k,y^k)\|^2 \notag \\
                &\qquad = c_1\|\nabla \psi(y^{k})\|^2 - (c_2\ell_{\sf yx}^2/m_{\sf x}^2)\|\nabla_xL(x^k,y^k)\|^2,
                \#
                where the first inequality is due to $2a^\intercal b\ge -\eta\|a\|^2 -\|b\|^2/\eta$ for any $a,b\in\RR^{p}$ and $\eta>0$, the last inequality is due to Lemma \ref{lem:wx_sc} and \eqref{eq:t_rho_pyL}, and the last equality is due to the definitions of constants $\eta_1$, $\eta_2$, $c_1$, and $c_2$. Moreover, when $\gamma_t = \Gamma_t$, it is easy to check that $\Gamma_t^2 - 2\rho \gamma_t + \rho^2=0$ and thus $\|t^k- \rho \nabla_{y}L(x^k,y^k)\|^2=0$ in \eqref{eq:t_rho_pyL}. Thus, the middle term on the second line in \eqref{eq:lambdainner} is zero. Thus, we can take $\eta_1=0$ and $\eta_2=\rho$ and verify that \eqref{eq:lambdainner} still holds. Furthermore, we lower bound the last term in \eqref{eq:step1} by Lemma \ref{lem:de_lambda},
                \#\label{eq:step1_last}
                - (\ell_\psi/2)\|y^{k+1} - y^k\|^2 \ge - \beta^2 \ell_\psi \Gamma_t^2 \|\nabla \psi(y^k)\|^2 - (\beta^2\ell_\psi\Gamma_t^2\ell_{\sf yx}^2/m_{\sf x}^2)\|\nabla_xL(x^k,y^k)\|^2.
                \#
                Thus, by substituting \eqref{eq:lambdainner} and \eqref{eq:step1_last} into \eqref{eq:step1}, we have, 
                \$
                \psi(y^{k+1}) &\ge \psi(y^k) +(c_1-\beta \ell_\psi\Gamma_t^2)\beta\|\nabla \psi(y^{k})\|^2 - [(c_2+\beta\ell_\psi\Gamma_t^2){\beta\ell_{\sf yx}^2}/{m_{\sf x}^2}] \|\nabla_xL({x}^k,y^k)\|^2.
                \$
                Therefore, by subtracting $\Xi_\psi$ and taking negative signs on both sides, we conclude the proof.
                \end{proof}

\begin{proof}[Proof of Proposition \ref{prop:Lg}]
We split $x$'s updated tracking error $\Delta_{x}^{k+1}$ defined in \eqref{eq:terrors} as follows,
\#\label{eq:deltaL}
& \Delta_{x}^{k+1} =  L(x^{k+1}, y^{k+1}) - L(x^*(y^{k+1}), y^{k+1}) \\
&\ = \underbrace{L(x^{k+1}, y^{k+1}) -L(x^{k+1}, y^{k})}_{\textstyle \text{term (A)}}+ \underbrace{L(x^{k+1}, y^{k}) -L(x^*(y^{k}), y^{k})}_{\textstyle \text{term (B)}} + \underbrace{L(x^*(y^{k}), y^{k}) - L(x^*(y^{k+1}), y^{k+1})}_{\textstyle \text{term (C)}}.\notag
\#
Here term (A) measures the change due to the $y$-update, term (B) characterizes $x$'s updated tracking error, and term (C) shows the difference between $y$'s optimality measures.We now upper bound terms (A)-(C).
                    
\noindent\textit{Term (A).} By the $\ell_{\sf yy}$-Lipschitz continuity of $\nabla_yL(x^{k+1},y)$ with respect to $y$ under Assumption \ref{ass:Hessian}, we have
\#\label{eq:termA}
& L(x^{k+1}, y^{k+1}) -L(x^{k+1}, y^{k}) \le (y^{k+1}-y^{k})^\intercal  \nabla_yL(x^{k+1}, y^k)+ \frac{\ell_{\sf yy}}{2}\|y^{k+1} - y^k\|^2.
\#
The last term in \eqref{eq:termA} is upper bounded by the $y$-update in Algorithm \ref{alg:grand} and Lemma \ref{lem:descent_direct} as follows,
\#\label{eq:termA_y}
\|y^{k+1} - y^k\|^2 = \beta^2\|t^k\|^2 \le {\beta^2\Gamma_t^2}\|\nabla_yL(x^k,y^k)\|^2.
\#
Next, we upper bound the first term in \eqref{eq:termA}. By the $y$-update in Algorithm \ref{alg:grand}, we have
\$
(y^{k+1}-y^{k})^\intercal \nabla_yL(x^{k+1}, y^k) 
&= \beta(t^k)^\intercal \nabla_yL(x^{k+1}, y^k)\notag\\
& = \beta(t^k)^\intercal\nabla_y L(x^k,y^k) + \beta(t^k)^\intercal \big[\nabla_yL(x^{k+1}, y^k) - \nabla_y L(x^k,y^k)\big] \notag\\
&\le \beta\|t^k\|\|\nabla_y L(x^k,y^k)\| + \frac{\beta}{2}\|t^k\|^2 + \frac{\beta}{2}\|\nabla_yL(x^{k+1}, y^k) - \nabla_y L(x^k,y^k)\|^2\ \notag\\
& \le \beta\Gamma_t \big(1 + \frac{\Gamma_t}{2}\big) \|\nabla_y L(x^k,y^k)\|^2 + \frac{\beta\ell_{\sf yx}^2}{2}\|{x}^{k+1} - {x}^{k}\|^2, 
\$
where the first inequality holds since $2a^\intercal b\le \|a\|^2 +\|b\|^2$ for any $a,b\in\RR^{p}$ and the last inequality is due to the $\ell_{\sf yx}$-Lipschitz continuity of $\nabla_yL(x,y)$ relative to $x$ under Assumption \ref{ass:Hessian}. The last term in the preceding relation is bounded by the $x$-update in Algorithm \ref{alg:grand} and Lemma \ref{lem:descent_direct} as follows,
\$
\|x^{k+1} - x^k\|^2 &=  \alpha^2\|s^k\|^2\le \alpha^2\Gamma_s^2\|\nabla_x L(x^k,y^k)\|^2.
\$
By substituting the above inequality into the previous one, the first term in \eqref{eq:termA} is bounded as,
\#\label{eq:termA_inter}
&(y^{k+1}-y^{k})^\intercal \nabla_yL(x^{k+1}, y^k) \le \beta\Gamma_t \big(1 + {\Gamma_t}/{2}\big) \|\nabla_y L(x^k,y^k)\|^2  + ({\alpha^2\beta\Gamma_s^2\ell_{\sf yx}^2}/{2})\|\nabla_x L(x^k,y^k)\|^2.
\#
Thus, by substituting \eqref{eq:termA_y}, \eqref{eq:termA_inter}, and constant $\iota_1$ into \eqref{eq:termA}, we bound term (A) as follows,
\#\label{eq:terma}
 L(x^{k+1}, y^{k+1}) -L(x^{k+1}, y^{k}) &\le \frac{\beta\iota_1}{2}\|\nabla_y L(x^k,y^k)\|^2 + \frac{\alpha^2\beta\Gamma_s^2\ell_{\sf yx}^2}{2}\|\nabla_x L(x^k,y^k)\|^2 \notag\\
&\le \beta\iota_1\Big[\frac{\ell_{\sf yx}^2}{m_{\sf x}^2}\|\nabla_xL(x^k,y^k)\|^2 + \|\nabla \psi(y^k)\|^2\Big] + \frac{\alpha^2\beta\Gamma_s^2\ell_{\sf yx}^2}{2}\|\nabla_x L(x^k,y^k)\|^2 \notag\\
&= \big(\frac{\alpha^2\Gamma_s^2}{2} + \frac{\iota_1}{m_{\sf x}^2}\big)\beta\ell_{\sf yx}^2\|\nabla_x L(x^k,y^k)\|^2 + \beta\iota_1 \|\nabla \psi(y^k)\|^2,
\#
where the last inequality follows from Lemma \ref{lem:de_lambda} with $\upsilon=1$. 

\noindent\textit{Term (B).} The $\ell_{\sf xx}$-Lipschitz continuity of $\nabla_x L(x,y^k)$ relative to $x$ under Assumption \ref{ass:Hessian} gives that
\$
L(x^{k+1}, y^{k}) &\le L(x^{k}, y^{k}) + \nabla_x L(x^{k}, y^{k})^\intercal(x^{k+1} - x^k) + \frac{\ell_{\sf xx}}{2}\|x^{k+1} - x^k\|^2\notag\\
&= L(x^{k}, y^{k}) - \alpha\nabla_x L(x^{k}, y^{k})^\intercal s^k + \frac{\alpha^2\ell_{\sf xx}}{2}\|s^k\|^2 \notag \\
&\le L(x^{k}, y^{k}) - \alpha\big(\gamma_s-\frac{\alpha\ell_{\sf xx}\Gamma_s^2}{2}\big)\|\nabla_x L(x^{k}, y^{k})\|^2,
\$
where the equality is due to the $x$-update in Algorithm \ref{alg:grand} and the last inequality is due to Lemma \ref{lem:descent_direct}.
By subtracting $L(x^{*}(y^k), y^k)$ on both sides of the preceding relation, we have an upper bound on term (B),
\#\label{eq:termb}
L(x^{k+1}, y^{k}) - L(x^{*}(y^k), y^k)\le \Delta_{x}^k - \alpha(\gamma_s-{\alpha\ell_{\sf xx}\Gamma_s^2}/{2})\|\nabla_x L(x^{k}, y^{k})\|^2.
\# 
 
 \noindent\textit{Term (C).} By definitions of the function $\psi$ in \eqref{prob:L} and $y$'s optimality measure $\Delta_{y}^k$ in \eqref{eq:terrors}, we have
 \#\label{eq:termc}
 L(x^*(y^{k}), y^{k}) - L(x^*(y^{k+1}), y^{k+1}) = \psi(y^k) - \psi(y^{k+1}) = \Delta_{y}^k - \Delta_{y}^{k+1}. 
 \#
 Finally, by substituting \eqref{eq:terma}, \eqref{eq:termb}, and \eqref{eq:termc} into \eqref{eq:deltaL}, we have
 \$
 \Delta_{x}^{k+1} &\le \Delta_{x}^k + \beta\iota_1 \|\nabla \psi(y^k)\|^2 +\Delta_{y}^k - \Delta_{y}^{k+1}  - \Big\{\alpha\big[\gamma_s-\frac{\alpha\Gamma_s^2}{2}(\ell_{\sf xx} + \beta\ell_{\sf yx}^2)\big] - \frac{\beta\iota_1\ell_{\sf yx}^2}{m_{\sf x}^2} \Big\} \|\nabla_x L(x^k,y^k)\|^2 . 
 \$
 By the definition of $\iota_2$, we conclude the proof of the proposition.
 \end{proof}
                    
 \begin{proof}[Proof of Proposition \ref{prop:pre}]
 By multiplying Lemma \ref{prop:gL} by $3\iota/c_1$ and adding Lemma \ref{prop:Lg}, we have
 \#\label{eq:step_1_and_2}
 \big(\frac{3\iota}{c_1}+1\big)\Delta_{y}^{k+1} + \Delta_{x}^{k+1} &\le  \big(\frac{3\iota}{c_1}+1\big)\Delta_{y}^{k} - \beta\big(3\iota - \frac{3\iota\beta\ell_\psi\Gamma_t^2}{c_1} - \iota_1\big) \|\nabla \psi(y^{k})\|^2 \notag \\
 &\quad + \Delta_{x}^k -\big[\iota_2 - \frac{3\iota\beta\ell_{\sf yx}^2}{c_1m_{\sf x}^2}({c_2} + {\beta\ell_\psi\Gamma_t^2} )\big]\|\nabla_xL({x}^k,y^k)\|^2.
 \#
 Now we lower bound the coefficients in \eqref{eq:step_1_and_2} by the stepsize conditions in Theorem \ref{thm:nonconvex}. First, we note that when $\beta\le c_1/[3\ell_\psi\Gamma_t^2]$ as required, using the definition of $\iota_1$, 
 we have $\iota_1\le\iota$. Thus,
 \#\label{eq:kappa_3}
 3\iota - {3\iota\beta\ell_\psi\Gamma_t^2}/{c_1} - \iota_1 \ge 2\iota - {3\iota\beta\ell_\psi\Gamma_t^2}/{c_1} \ge \iota,
 \#
 where the last inequality is due to $\beta\le c_1/[3\ell_\psi\Gamma_t^2]$. Moreover, the definition of $\iota_2$, we have
 \#\label{eq:kappa_4}
 \iota_2 - \frac{3\iota\beta\ell_{\sf yx}^2}{c_1m_{\sf x}^2}({c_2} + {\beta\ell_\psi\Gamma_t^2} )& = \alpha\Big[\gamma_s-\frac{\alpha\Gamma_s^2}{2}(\ell_{\sf xx} + \beta\ell_{\sf yx}^2)\Big] - \frac{\beta\ell_{\sf yx}^2}{m_{\sf x}^2} \big(\iota_1 + \frac{3\iota c_2}{c_1} + \frac{3\iota\beta\ell_\psi\Gamma_t^2}{c_1} \big) \notag\\
 & \ge \alpha\Big[\gamma_s-\frac{\alpha\Gamma_s^2}{2}\big(\ell_{\sf xx} + \frac{c_1\ell_{\sf yx}^2}{3\ell_\psi\Gamma_t^2}\big)\Big] - \frac{\beta\ell_{\sf yx}^2\iota}{m_{\sf x}^2}  \big(\frac{3c_2}{c_1} + 2\big)  \ge \frac{\alpha\gamma_s}{3}, 
 \#
 where the first inequality is due to $\beta\le c_1/(3\ell_\psi\Gamma_t^2)$ and thus $\iota_1\le\iota$ and the last inequality is due to the stepsize conditions.
 By substituting \eqref{eq:kappa_3} and \eqref{eq:kappa_4} into \eqref{eq:step_1_and_2}, we have
 \#\label{eq:pre_thm}
 \big(\frac{3\iota}{c_1}+1\big)\Delta_{y}^{k+1} + \Delta_{x}^{k+1} &\le  \big(\frac{3\iota}{c_1}+1\big)\Delta_{y}^{k} - \beta\iota \|\nabla \psi(y^{k})\|^2 + \Delta_{x}^k - \frac{\alpha\gamma_s}{3}\|\nabla_xL({x}^k,y^k)\|^2.
 \#
 Due to the $m_{\sf x}$-strong convexity of $L(\cdot, y^k)$ under Assumption \ref{ass:Hessian}, we have
 \$
 \|\nabla_x L({x}^{k}, y^{k})\|^2 \ge 2m_{\sf x}\big(L(x^k, y^k) - L(x^*(y^k), y^k)\big) = 2m_{\sf x}\Delta_{x}^k.
 \$
 Thus, by substituting the above relation into \eqref{eq:pre_thm}, we conclude the proof.
 \end{proof}

\begin{proof}[Proof of Theorem \ref{thm:nonconvex}]
Telescoping the inequalities in Proposition \ref{prop:pre} from $k=0$ to $K-1$, we have
\$
\big({3\iota}/{c_1}+1\big)\Delta_{y}^{K} + \Delta_{x}^{K} \le \big({3\iota}/{c_1}+1\big)\Delta_{y}^0 +\Delta_{x}^0 - \beta\iota\sum_{k=0}^{K-1}\|\nabla \psi(y^{k})\|^2 - ({2\alpha\gamma_sm_{\sf x}}/{3})\sum_{k=0}^{K-1}\Delta_{x}^k.
\$
By rearranging the terms in the above relationship, we have 
\$
\frac{1}{K}\sum_{k=0}^{K-1}\Big[\beta\iota\|\nabla \psi(y^{k})\|^2 + \frac{2\alpha\gamma_sm_{\sf x}}{3}\Delta_{x}^k\Big] & \le \frac{1}{K} \Big[\big(\frac{3\iota}{c_1}+1\big)\Delta_{y}^0 +\Delta_{x}^0-\big(\frac{3\iota}{c_1}+1\big)\Delta_{y}^{K} - \Delta_{x}^{K}\Big] \\
& \le \frac{1}{K} \Big[\big(\frac{3\iota}{c_1}+1\big)\Delta_{y}^0 +\Delta_{x}^0\Big].
\$
The last inequality holds since $\Delta_y^K\ge 0$ and $\Delta_x^K\ge 0$. This concludes the proof of the theorem.
\end{proof}

Before showing Examples \ref{ex:g-h} and \ref{ex:full-rank}, we first consider the structured minimax problem with $L$ defined in \eqref{eq:structured_minimax}. Since $\nabla_{xx}^2 L(x,y) = \nabla^2f(x)$ in \eqref{eq:structured_minimax}, the $m_{\sf x}$-strong convexity of $L(x,y)$ on $x$ and the $\ell_{\sf xx}$-Lipschitz continuity of $\nabla_x L(x,y)$ on $x$ under Assumption \ref{ass:Hessian} imply that $f(x)$ is $m_{\sf x}$-strongly convex and $\nabla f(x)$ is $\ell_{\sf xx}$-Lipschitz continuous, respectively.
Thus, the conjugate function $f^*(\lambda)$ is $(1/\ell_{\sf xx})$-strongly convex and $(1/m_{\sf x})$-Lipschitz continuous \cite{hiriart2013convex}. Moreover, we note that $\psi(y)$ defined in \eqref{eq:x_star_y} can be written as follows when $L$ is of the structured form \eqref{eq:structured_minimax},
\#\label{eq:structured-psi}
\psi(y) = -f^*(-W^\intercal y) - g(y).
\#
Now we will show that Assumption \ref{ass:psi_PL} holds for Examples \ref{ex:g-h} - \ref{ex:full-rank} in the sequel.
    
\begin{proof} [Proof of Example \ref{ex:g-h}]  If there exists a function $h:\RR^d \to\RR$ such that $g(y) = h(W^\intercal y)$, following from \eqref{eq:structured-psi}, we have
    \$
    \psi(y) = - f^*(-W^\intercal y) - h(W^\intercal y).
    \$
    If we define a function $\psi:\RR^d\to\RR$ such that $\phi = - f^* - h$, we have $\psi(y) = \phi(W^\intercal y)$. Now we prove that $\phi$ is $(1/\ell_{\sf xx} - m_h)$-strongly concave by showing the concavity of the function $\phi(\lambda) - (1/\ell_{\sf xx} - m_h)\|\lambda\|^2/2$. We decompose the function as follows,
    \$
    \phi(\lambda) + \frac{1}{2}\big(\frac{1}{\ell_{\sf xx}} - m_h\big)\|\lambda\|^2= - \big[f^*(\lambda) - \frac{1}{2\ell_{\sf xx}}\|\lambda\|^2\big]-\big[h(\lambda) + \frac{m_h}{2}\|\lambda\|^2\big],
    \$
    where the function $f^*(\lambda) - \|\lambda\|^2/(2\ell_{\sf xx})$ is convex due to the $(1/\ell_{\sf xx})$-strong convexity of $f^*$ and it is assumed that the function $h(\lambda) + m_h\|\lambda\|^2/2$ is convex. Therefore, the function $\phi(\lambda) - (1/\ell_{\sf xx} - m_h)\|\lambda\|^2/2$ is concave and thus $\phi$ is $(1/\ell_{\sf xx} - m_h)$-strongly concave. Since the function $\psi$ is a strongly concave $\phi$ composed with a linear mapping $W^\intercal$, it satisfies the PL inequality with $p_\psi = \sigma_{\min}^+(W)(1/\ell_{\sf xx} - m_h)$ \cite{karimi2016linear}.
    \end{proof}

\begin{proof}[Proof of Example \ref{ex:side-PL}] 
        For any $y$, Lemma \ref{lem:psi_property} and the one-sided PL condition at $(x^*(y), y)$ yield
        \$
        \big\|\nabla \psi(y)\big\|^2 &= \|\nabla_y L(x^*(y),y)\|^2 \\
        &\ge 2p_\psi[\max_{\tilde y} L(x^*(y),{\tilde y}) - L(x^*(y),y)] \\
        & \ge 2p_\psi[L(x^*(y),y^*) - L(x^*(y),y)] \\
        & \ge 2p_\psi[L(x^*(y^*),y^*) - L(x^*(y),y)] \\
        & = 2p_\psi[\psi(y^*) - \psi(y)],
        \$
        where the second inequality holds due to $\max_{\tilde y} L(x^*(y),{\tilde y})\ge L(x^*(y),\hat y)$ for any $\hat y\in\RR^p$ and the last inequality is due to $x^*(y^*) =\argmin_x L(x,y^*)$. 
        \end{proof}
    
\begin{proof}[Proof of Example \ref{ex:full-rank}] We calculate $\nabla^2 \psi(y)$ based on \eqref{eq:structured-psi} as follows,
    \#\label{eq:casec}
    \nabla^2\psi(y) = - W[\nabla^2f^*(W^\intercal y)]W^\intercal - \nabla^2g(y).
    \#
    Since $\rank(WW^\intercal) = \rank(W^\intercal) =\rank(W)= p$ and $WW^\intercal\in\RR^{p\times p}$, the matrix $WW^\intercal$ has full rank and thus $\sigma_{\min}(W) >0$. For any $\omega\in\RR^p$, since $\nabla^2f^*(W^\intercal y)\succeq (1/\ell_{\sf xx})I_p$, we have
    \$
    \omega^\intercal W[\nabla^2f^*(W^\intercal y)]W^\intercal \omega \ge \frac{1}{\ell_{\sf xx}} \omega^\intercal WW^\intercal \omega \ge \frac{\sigma_{\min}^2(W)}{\ell_{\sf xx}}\|\omega\|^2.
    \$
    Therefore, we have $W[\nabla^2f^*(W^\intercal y)]W^\intercal \succeq [\sigma_{\min}^2(W)/\ell_{\sf xx}]I_p$. Moreover, we have $\nabla^2g(y) \succeq -m_gI_p$ since $g(y) + m_g\|y\|^2/2$ is convex. Thus, following from \eqref{eq:casec}, we have
    \$
    \nabla^2\psi(y) \preceq - \big(\frac{\sigma_{\min}^2(W)}{\ell_{\sf xx}} - m_g\big) I_p.
    \$
    This implies that $\psi(y)$ is $[\sigma_{\min}^2(W)/\ell_{\sf xx}-m_g]$-strongly concave. Therefore, $\psi(y)$ satisfies Assumption \ref{ass:psi_PL} with $p_\psi = \sigma_{\min}^2(W)/\ell_{\sf xx}-m_g$ \cite{karimi2016linear}.
\end{proof}

\begin{proof}[Proof of Theorem \ref{thm:PL}]
    Substituting the PL condition in Assumption \ref{ass:psi_PL} into Proposition \ref{prop:pre}, we have
    \$
    ({3\iota}/{c_1}+1)\Delta_{y}^{k+1} + \Delta_{x}^{k+1} &\le  ({3\iota}/{c_1}+1- 2\beta\iota p_\psi)\Delta_{y}^{k}+ (1- {2\alpha\gamma_sm_{\sf x}}/{3} )\Delta_{x}^k \\
    & = [1 - {2\beta\iota p_\psi c_1}(3\iota + c_1)]({3\iota}/{c_1}+1)\Delta_{y}^{k}+ (1- {2\alpha\gamma_sm_{\sf x}}/{3})\Delta_{x}^k \\
    & \le (1-\delta) [({3\iota}/{c_1}+1)\Delta_{y}^{k} + \Delta_{x}^{k}],
    \$
    where $\delta$ is defined in \eqref{eq:delta}. The stepsize conditions with $\gamma_s \le \Gamma_s$ in Lemma \ref{lem:descent_direct} and $m_{\sf x}\le \ell_{\sf xx}$ give
    \$
    \alpha \le {2\gamma_s}/{[\Gamma_s^2(3\ell_{\sf xx} + c_1\ell_{\sf yx}^2/\ell_\psi\Gamma_t^2)]} \le{2\gamma_s}/(3\Gamma_s^2\ell_{\sf xx}) \le {2}/(3\gamma_s m_{\sf x})  < {3}/(2\gamma_s m_{\sf x}).
    \$
    Thus, we have $0<\delta<1$ when stepsizes $\alpha$ and $\beta$ satisfy the required conditions.
    \end{proof}
    Proofs of Theorems \ref{thm:nonconvex-alt} and \ref{thm:PL-alt} follows the same steps as those of Theorems \ref{thm:nonconvex} and \ref{thm:PL}.

\section{Proof of Theorem \ref{thm:uj-local-quad}}
\begin{proof}[Proof of Theorem \ref{thm:uj-local-quad}]
    We note that $U_J$ is twice differentiable under Assumption \ref{ass:Hessian}.
    Moreover, it is easy to check that $X(x^\dagger, y^\dagger) = (x^\dagger, y^\dagger)$ and $Y(x^\dagger, y^\dagger) = (x^\dagger, y^\dagger)$, and thus $U_J(x^\dagger, y^\dagger)=(x^\dagger, y^\dagger)$, meaning that $(x^\dagger, y^\dagger)$ is a fixed point of $U_J$. Now we only need to show that $U_J^\prime(x^\dagger, y^\dagger)=0$ due to Lemma \ref{lem:ortega}. For convenience, we define the derivatives $X^\prime_* = X^\prime(x^\dagger,y^\dagger)$ and $Y^\prime_* = Y^\prime(x^\dagger,y^\dagger)$. By using first-order stationarity of $(x^\dagger,y^\dagger)$, we obtain that,
    \$
    & X^\prime_* = \begin{pmatrix}
        0 & - [\nabla_{xx}^2L(x^\dagger,y^\dagger)]^{-1}\nabla_{xy}^2L(x^\dagger,y^\dagger) \\
        0 & I\end{pmatrix}, \\
    & Y^\prime_* = \begin{pmatrix}
        I & 0 \\
        [N(x^\dagger,y^\dagger)]^{-1}\nabla_{yx}^2L(x^\dagger,y^\dagger) & I + [N(x^\dagger,y^\dagger)]^{-1}\nabla_{yy}^2L(x^\dagger,y^\dagger)
    \end{pmatrix}.
    \$
    Based on the above formula, it is easy to show that $[X^\prime_*]^2 = X^\prime_*$, meaning that $X^\prime_*$ is idempotent.
    Since $(x^\dagger, y^\dagger)$ is a fixed point of $X$ and $Y$, by the chain rule and the idempotence of $X^\prime_*$, we have
    \#\label{eq:U-prime}
    U_J^\prime(x^\dagger, y^\dagger) = (X_*^\prime)^J Y_*^\prime X_*^\prime = X_*^\prime Y_*^\prime X_*^\prime.
    \#
    Straightforward algebraic manipulation yields that
    \$
    &Y^\prime_* X^\prime_* \\
    & = \begin{pmatrix}
        0 & - [\nabla_{xx}^2L(x^\dagger,y^\dagger)]^{-1}\nabla_{xy}^2L(x^\dagger,y^\dagger) \\
        0 & - [N(x^\dagger,y^\dagger)]^{-1}\nabla_{yx}^2L(x^\dagger,y^\dagger)[\nabla_{xx}^2L(x^\dagger,y^\dagger)]^{-1}\nabla_{xy}^2L(x^\dagger,y^\dagger) + I + [N(x^\dagger,y^\dagger)]^{-1}\nabla_{yy}^2L(x^\dagger,y^\dagger)\end{pmatrix} \\
        &= \begin{pmatrix}
            0 & - [\nabla_{xx}^2L(x^\dagger,y^\dagger)]^{-1}\nabla_{xy}^2L(x^\dagger,y^\dagger) \\
            0 & 0\end{pmatrix},
    \$
    where the last equality holds due to the definition of $N$ in \eqref{eq:N-operator}. 
    Then by further calculation, we have $X_*^\prime Y_*^\prime X_*^\prime = 0$. Thus, by substituting it into \eqref{eq:U-prime}, we have
    $
    U_J^\prime(x^\dagger, y^\dagger) = X_*^\prime Y_*^\prime X_*^\prime = 0.
    $
     Thus, the iterates generated by $U_J$ converge locally in at least a quadratic rate due to Lemma \ref{lem:ortega}.
\end{proof}

\end{document}